\documentclass[A4]{article}
\usepackage{amscd}
\usepackage{amssymb}
\usepackage{amsmath}
\usepackage{amsthm}
\usepackage{latexsym}
\usepackage{upref}
\usepackage{epsfig}
\usepackage{mathrsfs}
\usepackage{float}
\usepackage{indentfirst}
\usepackage{hyperref}
\hypersetup{
    colorlinks=true,
    linkcolor=blue,
    citecolor=blue,
    filecolor=blue,
    urlcolor=blue,
}
\usepackage{graphics,graphicx,color}
\usepackage{floatflt}
\usepackage{cite}
\usepackage{enumitem}
\usepackage{fancyhdr} %%%Running Title

 \newtheorem{thm}{Theorem}[section]
 \newtheorem{cor}[thm]{Corollary}
 \newtheorem{lem}[thm]{Lemma}
 \newtheorem{prop}[thm]{Proposition}
 \theoremstyle{definition}
 \newtheorem{defn}[thm]{Definition}
 \theoremstyle{remark}
 \newtheorem*{rem}{Remark}
 
 \numberwithin{equation}{section}

\makeatletter
\def\th@plain{%
  \thm@notefont{}% same as heading font
  \itshape % body font
}
\def\th@definition{%
  \thm@notefont{}% same as heading font
  \normalfont % body font
} \makeatother
\setlist{font=\normalfont}

\newcommand{\Gam}{\Gamma}
\newcommand{\alp}{\alpha}
\newcommand{\lamb}{\lambda}
\newcommand{\bet}{\beta}
\newcommand{\lgyr}[2]{{\mathrm{lgyr}[{#1}]}{#2}}
\newcommand{\rgyr}[2]{{\mathrm{rgyr}[{#1}]}{#2}}
\newcommand{\gyr}[2]{{\mathrm{gyr}[{#1}]}{#2}}
\newcommand{\ilgyr}[2]{{\mathrm{lgyr^{-1}}[{#1}]}{#2}}
\newcommand{\irgyr}[2]{{\mathrm{rgyr^{-1}}[{#1}]}{#2}}

\newcommand{\aut}[1]{\mathrm{Aut}\,{(#1)}}
\newcommand{\id}[1]{\mathrm{id}_{#1}}
\newcommand{\set}[1]{\{#1\}}
\newcommand{\cset}[2]{\set{{#1}\colon{#2}}}
\newcommand{\BCset}[2]{\left\{#1\colon #2\right\}}
\newcommand{\Bset}[1]{\left\{#1\right\}}
\newcommand{\Bp}[1]{\left(#1\right)}
\newcommand{\R}{\mathbb{R}}
\newcommand{\N}{\mathbb{N}}
\newcommand{\so}[1]{\mathrm{SO}(#1)}
\newcommand{\spin}[1]{\mathrm{Spin}\,(#1)}
\newcommand{\CL}[1]{\mathrm{C}\ell\hskip1pt(#1)}
\newcommand{\Cg}[1]{\Gamma{\mathbin{(#1)}}}
\newcommand{\mulCL}[1]{\mathrm{C}\ell^\times{(#1)}}
\newcommand{\GL}[1]{\mathrm{GL}\,(#1)}
\newcommand{\Or}[1]{\mathrm{O}(#1)}

\renewcommand\footnotemark{}

\begin{document}
% ---------------------TITLE----------------------------------
\title{\textbf{Bi-gyrogroup: The Group-like Structure Induced by Bi-decomposition of Groups$^\star$}\footnote{$^\star$This is the final version of the manuscript published
in Mathematics Interdisciplinary Research {\bf 1} (2016), pp.
111--142. The published version of the article is accessible via
\href{http://mir.kashanu.ac.ir/article_13911_2153.html}{{http://mir.kashanu.ac.ir/}}.}}
\author{Teerapong Suksumran$^\mathrm{a, b,\ast, \dag}$\footnote{$^\ast$Corresponding author (E-mail address: teerapong.suksumran@gmail.com).}\footnote{$^\dag$The first author was financially supported by Institute for
Promotion of Teaching Science and Technology (IPST), Thailand, via
Development and Promotion of Science and Technology Talents Project
(DPST).}\,\, and Abraham A. Ungar$^\mathrm{a, \ddag}$\footnote{$^\ddag$E-mail address: abraham.ungar@ndsu.edu}\\
$^\mathrm{a}$Department of Mathematics\\
North Dakota State University\\
Fargo, ND 58105, USA\\[0.3cm]
$^\mathrm{b}$Department of Mathematics and Computer Science\\
Faculty of Science, Chulalongkorn University\\Bangkok 10330,
Thailand\\}
\date{}
\maketitle

%Abstract-----------------------------------------------------
\begin{abstract}
The decomposition $\Gamma=BH$ of a group $\Gamma$ into a subset $B$
and a subgroup $H$ of $\Gamma$ induces, under general conditions, a
group-like structure for $B$, known as a gyrogroup. The famous
concrete realization of a \mbox{gyrogroup}, which motivated the
emergence of gyrogroups into the mainstream, is the space of all
relativistically admissible velocities along with a binary
\mbox{operation} given by the Einstein velocity addition law of
special relativity theory. The latter leads to the Lorentz
transformation group $\so{1,n}$, $n\in\N$, in pseudo-Euclidean
spaces of signature $(1, n)$. The study in this article is motivated
by generalized Lorentz groups $\so{m, n}$, $m, n\in\N$, in
pseudo-Euclidean spaces of signature $(m, n)$. Accordingly, this
article explores the bi-decomposition $\Gamma = H_LBH_R$ of a group
$\Gamma$ into a subset $B$ and subgroups $H_L$ and $H_R$ of
$\Gamma$, along with the novel bi-gyrogroup structure of $B$ induced
by the bi-decomposition of $\Gamma$. As an example, we show by
methods of Clifford \mbox{algebras} that the quotient group of the
spin group $\spin{m, n}$ possesses the bi-decomposition structure.
\end{abstract}

\thispagestyle{empty}
\section{Introduction}\label{sec: introdcution}
\par Lorentz transformation groups $\Gam = \so{1,n}$, $n\in\N$,
possess the decomposition structure $\Gam = BH$, where $B$ is a
subset of $\Gam$ and $H$ is a subgroup of $\Gam$ \cite{AU1988TRP}.
The decomposition structure of $\Gam$ induces a group-like structure
for $B$. This group-like structure was discovered in 1988
\cite{AU1988TRP} and became known as a \textit{gyrogroup}
\cite{AU1991TPI, AU2001BEA}. Subsequently, gyrogroups turned out to
play a universal computational role that extends far beyond the
domain of Lorentz groups $\so{1,n}$ \cite{AU2010BCE, AU2015AHG}, as
noted by Chatelin in \cite[p.~523]{FC2012QCA} and in references
therein. In fact, gyrogroups are special loops that, according to
\cite{TJASRA2014SBM}, are placed centrally in loop theory.

\par The use of Clifford algebras to employ gyrogroups as a computational
tool in harmonic analysis is presented by Ferreira in the seminal
papers \cite{MF2014HAE, MF2015HAM}. The use of Clifford algebras to
obtain a better understanding of gyrogroups is found, for instance,
in \cite{MF2009FMG, MF2011GPH, MFGR2011MGC, JL2010CAM, TSKW2015EGB}.

\par Generalized Lorentz transformation groups $\Gam = \so{m,n}$,
$m, n\in\N$, possess the so-called \textit{bi-decomposition}
structure $\Gam = H_LBH_R$, where $B$ is a subset of $\Gam$ and
$H_L$ and $H_R$ are subgroups of $\Gam$. The bi-decomposition
structure of $\Gam$ induces a group-like structure for $B$, called a
\textit{bi-gyrogroup} \cite{AU2015PRL}. The use of Clifford algebras
that may improve our understanding of bi-gyrogroups is found in
\cite{FMSF2013CBH}. Clearly, the notion of bi-gyrogroups extends the
notion of gyrogroups. Accordingly, ``gyro-language'', the algebraic
language crafted for gyrogroup theory is extended to
``bi-gyro-language'' for bi-gyrogroup theory.

\par As a first step towards demonstrating that bi-gyrogroups play a
universal computational role that extends far beyond the domain of
generalized Lorentz groups $\so{m, n}$, the aim of the present
article is to approach the study of bi-gyrogroups from the abstract
viewpoint.

\par The article is organized as follows. In Section \ref{sec: bigyrogroupoid}
we give the definition of a bi-gyrogroupoid. In Section \ref{sec:
bitransversal} we show that the bi-transversal decomposition of a
group with additional \mbox{properties} yields a highly structured
type of bi-gyrogroupoids. In Section \ref{sec: bigyrodecomposition}
we introduce the notion of bi-gyrodecomposition of groups and prove
that any bi-gyrodecomposition of a group gives rise to a
bi-gyrogroup. Finally, in Sections \ref{sec: pseudo-orthogonal
group} and \ref{sec: spin group} we demonstrate that the
pseudo-orthogonal group $\so{m, n}$ and the quotient group of the
spin group $\spin{m, n}$ possess the bi-gyrodecomposition structure.

\section{Bi-gyrogroupoids}\label{sec: bigyrogroupoid}

\par We begin with the abstract definition of a bi-gyrogroupoid,
which is modeled on the groupoid $\R^{n\times m}$ of all $n\times m$
real matrices with bi-gyroaddition studied in detail in
\cite{AU2015PRL}. We recall that a groupoid $(B, \oplus_b)$ is a
non-empty set $B$ with a binary operation $\oplus_b$. An
automorphism of a groupoid $(B, \oplus_b)$ is a bijection from $B$
to itself that preserves the groupoid operation. The group of all
automorphisms of $(B, \oplus_b)$ is denoted by $\aut{B, \oplus_b}$
or simply $\aut{B}$.

\begin{defn}[Bi-gyrogroupoid]\label{def: bigyrogroupoid}
A groupoid $(B, \oplus_b)$ is a \textit{bi-gyrogroupoid} if its
binary operation satisfies the following axioms.\\
(BG1) There is an element $0\in B$ such that $0\oplus_b a =
a\oplus_b 0 = a$ for all $a\in B$.\\
(BG2) For each $a\in B$, there is an element $b\in B$ such
that $b\oplus_b a = 0$.\\
(BG3) Each pair of $a$ and $b$ in $B$ corresponds to a left
automorphism $\lgyr{a, b}{}$
    and a right automorphism $\rgyr{a, b}{}$ in $\aut{B,\oplus_b}$ such
    that for all $c\in B$,
    \begin{equation}\label{eqn: bi-gyroassociative law}
    (a\oplus_b b)\oplus_b\lgyr{a, b}{c} = \rgyr{b, c}{a}\oplus_b(b\oplus_b
    c).
    \end{equation}
(BG4) For all $a, b\in B$,
    \begin{enumerate}
    \item[(a)] $\rgyr{a, b}{} = \rgyr{\lgyr{a, b}{a}, a\oplus_b b}{}$, and
    \item[(b)] $\lgyr{a, b}{} = \lgyr{\lgyr{a, b}{a}, a\oplus_b b}{}$.
    \end{enumerate}
(BG5) For all $a\in B$, $\lgyr{a, 0}{}$ and $\rgyr{a, 0}{}$ are the
identity map of $B$.
\end{defn}

\par A concrete realization of Axioms (BG1) through (BG5) will be presented
in Section \ref{sec: pseudo-orthogonal group}.

\par Roughly speaking, any bi-gyrogroupoid is a groupoid that comes
with two families of automorphisms, called left and right
automorphisms or, collectively, bi-automorphisms. Note that if
bi-automorphisms of a bi-gyrogroupoid $(B, \oplus_b)$ reduce to the
identity automorphism of $B$, then $(B, \oplus_b)$ forms a group.

\par Let $\ilgyr{a, b}{}$ and $\irgyr{a, b}{}$ be the inverse map
of $\lgyr{a, b}{}$ and $\rgyr{a, b}{}$, respectively. Let $\circ$
denote \textit{function composition} and let $\id{X}$ denote the
identity map on a non-empty set $X$. The following theorem asserts
that bi-gyrogroupoids satisfy a generalized associative law.

\begin{thm}\label{thm: left and right bi-gyroassociative law}
Any bi-gyrogroupoid $B$ satisfies the left bi-gyroassociative law
\begin{equation}\label{eqn: left bi-gyroassociative law}
a\oplus_b (b\oplus_b c) = (\irgyr{b, c}{a}\oplus_b
b)\oplus_b\lgyr{\irgyr{b, c}{a}, b}{c}
\end{equation}
and the right bi-gyroassociative law
\begin{equation}\label{eqn: right bi-gyroassociative law}
(a\oplus_b b)\oplus_b c = \rgyr{b, \ilgyr{a, b}{c}}{a}\oplus_b
(b\oplus_b \ilgyr{a, b}{c})
\end{equation}
for all $a, b, c\in B$.
\end{thm}
\begin{proof}
Let $a, b, c\in B$ be arbitrary. Since $\rgyr{b, c}{}$ is
surjective, there is an element $d\in B$ for which $\rgyr{b, c}{d} =
a$. By (BG3),
$$a\oplus_b (b\oplus_b c) = \rgyr{b, c}{d}\oplus_b(b\oplus_b c) = (d\oplus_b b)\oplus_b\lgyr{d, b}{c}.$$
Since $d = \irgyr{b, c}{a}$, \eqref{eqn: left bi-gyroassociative
law} is obtained. One obtains \eqref{eqn: right bi-gyroassociative
law} in a similar way.
\end{proof}

\begin{lem}\label{lem: Uniqueness of 0 in a bigyrogroup}
Any bi-gyrogroupoid $B$ has a unique two-sided identity element.
\end{lem}
\begin{proof}
By Definition \ref{def: bigyrogroupoid}, $B$ has a two-sided
identity element.  Suppose that $e$ and $f$ are two-sided identity
elements of $B$. As $e$ is a left identity, $e\oplus_b f = f$. As
$f$ is a right identity, $e\oplus_b f = e$. Hence, $e = e\oplus_b f
= f$.
\end{proof}

\par Following Lemma \ref{lem: Uniqueness of 0 in a bigyrogroup}, the
unique two-sided identity of a bi-gyrogroupoid will be denoted by
$0$. Let $B$ be a bi-gyrogroupoid and let $a\in B$. We say that
$b\in B$ is a \textit{left inverse} of $a$ if $b\oplus_b a = 0$ and
that $c\in B$ is a \textit{right inverse} of $a$ if $a\oplus_b c =
0$. To see that each element of a bi-gyrogroupoid has a unique
two-sided inverse, we investigate some basic properties of a
bi-gyrogroupoid.

\begin{thm}\label{thm: basic properties of bigyrogroup I}
Let $B$ be a bi-gyrogroupoid. The following properties are true.
\begin{enumerate}
    \item\label{item: bigyration of 0} For all $a, b\in B$, $\lgyr{a, b}{0} = 0$ and $\rgyr{a, b}{0} = 0$.
    \item\label{item: bigyration generated by a,a} For all $a\in B$, $\lgyr{a, a}{} = \id{B}$ and $\rgyr{a, a}{} =
    \id{B}$.
    \item\label{item: bigyration generated by b,a and b a left inverse of a} If $a$ is a left inverse of $b$, then $\lgyr{a, b}{} =
    \id{B}$ and $\rgyr{a, b}{} = \id{B}$.
    \item\label{item: left cancelation law} For all $b, c\in B$, if $a$ is a left inverse of
    $b$, then $\rgyr{b, c}{a}\oplus_b(b\oplus_b c) = c$.
    \item\label{item: left inverse is a right inverse} For all $a\in B$, if $b$ is a left inverse of $a$, then
    $b$ is a right inverse of $a$.
\end{enumerate}
\end{thm}
\begin{proof}
\eqref{item: bigyration of 0} Let $a, b\in B$. Let $c\in B$ be
arbitrary. Since $\lgyr{a, b}{}$ is surjective, $c = \lgyr{a, b}{d}$
for some $d\in B$. Then $$c\oplus_b\lgyr{a, b}{0} = \lgyr{a,
b}{d}\oplus_b\lgyr{a, b}{0} = \lgyr{a, b}{(d\oplus_b 0)} = \lgyr{a,
b}{d} = c.$$ Similarly, $(\lgyr{a, b}{0})\oplus_b c = c$. Hence,
$\lgyr{a, b}{0}$ is a two-sided identity of $B$. By Lemma \ref{lem:
Uniqueness of 0 in a bigyrogroup}, $\lgyr{a, b}{0} = 0$. Similarly,
one can prove that $\rgyr{a, b}{0} = 0$.

\par \eqref{item: bigyration generated by a,a} Setting $b = 0$ in
(BG4a) gives $\rgyr{a, a}{} = \rgyr{a, 0}{} = \id{B}$ by (BG5).
Similarly, setting $b = 0$ in (BG4b) gives $\lgyr{a, a}{} = \id{B}$.

\par \eqref{item: bigyration generated by b,a and b a left inverse of
a} Let $b\in B$ and let $a$ be a left inverse of $b$. By (BG4a) and
(BG5),
$$\rgyr{a, b}{} = \rgyr{\lgyr{a, b}{a}, a\oplus_b b}{} = \rgyr{\lgyr{a, b}{a}, 0}{} = \id{B}.$$
Similarly, $\lgyr{a, b}{} = \id{B}$ by (BG4b) and (BG5).

\par \eqref{item: left cancelation law} Let $b, c\in B$ and let $a$ be a left inverse of $b$.
From Identity \eqref{eqn: bi-gyroassociative law} and Item
\eqref{item: bigyration generated by b,a and b a left inverse of a},
we have $\rgyr{b, c}{a}\oplus_b(b\oplus_b c) = (a\oplus_b
b)\oplus_b\lgyr{a, b}{c} = 0\oplus_b c = c$.

\par \eqref{item: left inverse is a right inverse}. Let $a\in B$
and let $b$ be a left inverse of $a$. By (BG2), $b$ has a left
inverse, say $\tilde{b}$. From Items \eqref{item: left cancelation
law} and \eqref{item: bigyration generated by b,a and b a left
inverse of a}, we have
$$a = \rgyr{b, a}{\tilde{b}}\oplus_b(b\oplus_b a) = \rgyr{b, a}{\tilde{b}}\oplus_b 0 = \rgyr{b, a}{\tilde{b}} = \tilde{b}.$$
It follows that $a\oplus_b b = \tilde{b}\oplus_b b = 0$, which
proves $b$ is a right inverse of $a$.
\end{proof}

\begin{thm}\label{thm: unique two-sided inverse}
Any element of a bi-gyrogroupoid $B$ has a unique two-sided inverse
in $B$.
\end{thm}
\begin{proof}
Let $a\in B$. By (BG2), $a$ has a left inverse $b$ in $B$. By
Theorem \ref{thm: basic properties of bigyrogroup I} \eqref{item:
left inverse is a right inverse}, $b$ is also a right inverse of
$a$. Hence, $b$ is a two-sided inverse of $a$. Suppose that $c$ is a
two-sided inverse of $a$. Then $a$ is a left inverse of $c$. By
Theorem \ref{thm: basic properties of bigyrogroup I} \eqref{item:
bigyration generated by b,a and b a left inverse of a}--\eqref{item:
left cancelation law}, $c = \rgyr{a, c}{b}\oplus_b(a\oplus_b c) =
\rgyr{a, c}{b}\oplus_b 0 = \rgyr{a, c}{b} = b$, which proves the
uniqueness of $b$.
\end{proof}

\par Following Theorem \ref{thm: unique two-sided inverse}, if $a$ is
an element of a bi-gyrogroupoid, then the unique two-sided inverse
of $a$ will be denoted by $\ominus_b a$. We also write $a\ominus_b
b$ instead of $a\oplus_b(\ominus_b b)$. As a consequence of Theorems
\ref{thm: basic properties of bigyrogroup I} and \ref{thm: unique
two-sided inverse}, we derive the following theorem.

\begin{thm}\label{thm: properties of unique inverse}
Let $B$ be a bi-gyrogroupoid. The following properties are true for
all $a, b, c\in B$:
\begin{enumerate}
    \item\label{item: --a = a} $\ominus_b(\ominus_b a) = a$;
    \item\label{item: - commutes with bi-gyration} $\lgyr{a, b}{(\ominus_b c)} = \ominus_b\lgyr{a,
    b}{c}$ and $\rgyr{a, b}{(\ominus_b c)} = \ominus_b\rgyr{a,
    b}{c}$;
    \item\label{item: gyration of a and -a} $\lgyr{a,\ominus_b a}{} = \lgyr{\ominus_b a, a}{} = \rgyr{a,\ominus_b a}{} = \rgyr{\ominus_b a, a}{} =
    \id{B}$.
\end{enumerate}
\end{thm}

\par Any bi-gyrogroupoid satisfies a generalized cancellation law,
as shown in the following theorem.

\begin{thm}\label{thm: left-right cancellation law}
Any bi-gyrogroupoid $B$ satisfies the left cancellation law
\begin{equation}\label{eqn: left cancellation law}
\ominus_b\rgyr{a, b}{a}\oplus_b(a\oplus_b b) = b
\end{equation}
and the right cancellation law
\begin{equation}\label{eqn: right cancellation law}
(a\oplus_b b)\ominus_b\lgyr{a, b}{b} = a
\end{equation}
for all $a, b\in B$.
\end{thm}
\begin{proof}
Identity \eqref{eqn: left cancellation law} follows from Theorem
\ref{thm: basic properties of bigyrogroup I} \eqref{item: left
cancelation law} and Theorem \ref{thm: properties of unique inverse}
\eqref{item: - commutes with bi-gyration}. \mbox{Identity}
\eqref{eqn: right cancellation law} follows from (BG3) with $c
=\ominus_b b$.
\end{proof}

\begin{defn}[Bi-gyrocommutative bi-gyrogroupoid]\label{def: Bi-gyrocommutative bi-groupoid}
A bi-gyrogroupoid $B$ is \textit{bi-gyrocommutative} if it satisfies
the bi-gyrocommutative law
\begin{equation}\label{eqn: bigyrocommutative law}
a\oplus_b b = (\lgyr{a, b}{}\circ\rgyr{a, b}{})(b\oplus_b a)
\end{equation}
for all $a, b\in B$.
\end{defn}

\begin{defn}[Automorphic inverse property]\label{def: AIP}
A bi-gyrogroupoid $B$ has the \textit{automorphic inverse property}
if
$$\ominus_b (a\oplus_b b) = (\ominus_b a)\oplus_b(\ominus_b b)$$
for all $a, b\in B$.
\end{defn}

\begin{defn}[Bi-gyration inversion law]\label{def: ISP}
A bi-gyrogroupoid $B$ satisfies the \textit{bi-gyration inversion
law} if
$$\ilgyr{a, b}{} = \lgyr{b, a}{}\quad\textrm{and}\quad\irgyr{a, b}{} = \rgyr{b, a}{}$$
for all $a, b\in B$.
\end{defn}

\par Under certain conditions, the bi-gyrocommutative property and
the auto-morphic inverse property are equivalent, as shown in the
following theorem.

\begin{thm}\label{thm: to be bigyrocommutative}
Let $B$ be a bi-gyrogroupoid such that
\begin{enumerate}
    \item $\lgyr{a, b}{}\circ\rgyr{a, b}{} = \rgyr{a, b}{}\circ\lgyr{a,
    b}{}$;
    \item $\ilgyr{a, b}{} = \lgyr{\ominus_b b, \ominus_b a}{}$ and $\irgyr{a, b}{} = \rgyr{\ominus_b b,
    \ominus_b a}{}$;
    \item $\ominus_b(a\oplus_b b) = (\lgyr{a, b}{}\circ\rgyr{a, b}{})(\ominus_b b \ominus_b a)$
\end{enumerate}
for all $a, b\in B$. If $B$ is bi-gyrocommutative, then $B$ has the
automorphic inverse property. The converse is true if $B$ satisfies
the bi-gyration inversion law.
\end{thm}
\begin{proof}
Suppose that $B$ is bi-gyrocommutative and let $a, b\in B$. Then
$b\oplus_b a = (\lgyr{b, a}{}\circ\rgyr{b, a}{})(a\oplus_b b)$ and
hence
\begin{eqnarray}\label{eqn: in proof equivalent of bigyrocomm. and AIP}
\begin{split}
a\oplus_b b &= (\lgyr{b, a}{}\circ\rgyr{b, a}{})^{-1}(b\oplus_b
a)\\
{} &= (\irgyr{b, a}{}\circ\ilgyr{b, a}{})(b\oplus_b a)\\
{} &= (\rgyr{\ominus_b a, \ominus_b b}{}\circ\lgyr{\ominus_b a,
\ominus_b b}{})(b \oplus_b a)\\
{} &= (\lgyr{\ominus_b a, \ominus_b b}{}\circ\rgyr{\ominus_b a,
\ominus_b b}{})(b \oplus_b a)\\
{} &= \ominus_b(\ominus_b a\ominus_b b).
\end{split}
\end{eqnarray}
The extreme sides of \eqref{eqn: in proof equivalent of bigyrocomm.
and AIP} imply $\ominus_b (a\oplus_b b) = \ominus_b a\ominus_b b$
and so $B$ has the automorphic inverse property.
\par Suppose that $B$ satisfies the bi-gyration inversion law and let $a, b\in B$. As
in \eqref{eqn: in proof equivalent of bigyrocomm. and AIP}, we have
$$(\lgyr{a, b}{}\circ\rgyr{a, b}{})(b\oplus_b a) = \ominus_b(\ominus_b a\ominus_b b) = a\oplus_b b.$$
Hence, $B$ is bi-gyrocommutative.
\end{proof}

\section{Bi-transversal decomposition}\label{sec: bitransversal}
\par In this section we study the bi-decomposition $\Gam = H_LBH_R$
of a group $\Gam$ into a subset $B$ and subgroups $H_L$ and $H_R$ of
$\Gam$. The bi-decomposition $\Gam = H_LBH_R$ leads to a
bi-gyrogroupoid $B$, and under certain conditions, a group-like
structure for $B$, called a \textit{bi-gyrogroup}. Further, in the
special case when $H_L$ is the trivial subgroup of $\Gam$, the
bi-decomposition $\Gam = H_LBH_R$ descends to the decomposition
studied in \cite{TFAU2000IDG}. It turns out that the bi-gyrogroup
$B$ induced by the bi-decomposition of $\Gam$ forms a gyrogroup, a
rich algebraic structure extensively studied, for instance, in
\cite{AU2001BEA, AU2008FMG, AU2008AHG, AU2009AGS, MF2009FMG,
MFGR2011MGC, MF2014HAE, MF2015HAM, TSKW2015ITG, TSKW2014LTG,
TSKW2015EGB, NSAU2013TER, RLAY2013TRG}.

\begin{defn}[Bi-transversal]\label{def: Bitrasversal}
A subset $B$ of a group $\Gam$ is said to be a
\textit{bi-transversal} of subgroups $H_L$ and $H_R$ of $\Gam$ if
every element $g$ of $\Gam$ can be written uniquely as $g = h_\ell
bh_r$, where $h_\ell\in H_L$, $b\in B$, and $h_r\in H_R$.
\end{defn}

\par Let $B$ be a bi-transversal of subgroups $H_L$ and $H_R$ in a group $\Gam$.
For each pair of elements $b_1$ and $b_2$ in $B$, the product
$b_1b_2$ gives unique elements $h_\ell(b_1, b_2)\in H_L$, $b_1\odot
b_2\in B$, and $h_r(b_1, b_2)\in H_R$ such that
\begin{equation}\label{eqn: Unique decomposition}
b_1b_2 = h_\ell(b_1, b_2)(b_1\odot b_2)h_r(b_1, b_2).
\end{equation}
Hence, any bi-transversal $B$ of $H_L$ and $H_R$ gives rise to
\begin{enumerate}
    \item a binary operation $\odot$ in $B$, called the \textit{bi-transversal operation};
    \item a map $h_\ell\colon B\times B\to H_L$, called \textit{the left
    transversal map};
    \item a map $h_r\colon B\times B\to H_R$, called \textit{the
    right transversal map}.
\end{enumerate}
The pair $(B, \odot)$ is called the \textit{bi-transversal groupoid
of $H_L$ and $H_R$}.

\par We will see shortly that the left and right transversal maps of
the bi-transversal groupoid $(B, \odot)$ generate automorphisms of
$(B, \odot)$, called \textit{left} and \textit{right}
\textit{gyrations} or, collectively, \textit{bi-gyrations}.
Accordingly, left and right gyrations are also called \textit{left}
and \textit{right} \textit{gyroautomorphisms}.

\begin{defn}[Bi-gyration]\label{def: Bigyration}
Let $B$ be a bi-transversal of subgroups $H_L$ and $H_R$ in a group
$\Gam$. Let $h_\ell$ and $h_r$ be the left and right transversal
maps, respectively. The \textit{left gyration $\lgyr{b_1, b_2}{}$ of
$B$ generated by $b_1, b_2\in B$} is defined by
\begin{equation}
\lgyr{b_1, b_2}{b} = h_r(b_1, b_2)bh_r(b_1, b_2)^{-1},\quad b\in B.
\end{equation}
The \textit{right gyration $\rgyr{b_1, b_2}{}$ of $B$ generated by
$b_1, b_2\in B$} is defined by
\begin{equation}
\rgyr{b_1, b_2}{b} = h_\ell(b_1, b_2)^{-1}bh_\ell(b_1, b_2),\quad
b\in B.
\end{equation}
\end{defn}

\begin{rem}
\par In Definition \ref{def: Bigyration}, left gyrations
are associated with the right transversal map $h_r$, and right
gyrations are associated with the left transversal map $h_\ell$.
\end{rem}

\par We use the convenient notation $x^h = hxh^{-1}$
and denote \textit{conjugation by $h$} by $\alp_h$. That is,
$\alp_h(x) = x^h = hxh^{-1}$. With this notation, the left and right
gyrations in Definition \ref{def: Bigyration} read
\begin{equation}\label{eqn: bigyration in conjugate form}
\lgyr{a, b}{} = \alp_{h_r(a, b)}\quad\textrm{ and }\quad \rgyr{a,
b}{} = \alp_{h_\ell(a, b)^{-1}}
\end{equation}
for all $a, b\in B$. Let $B$ be a non-empty subset of a group
$\Gam$. We say that a subgroup $H$ of $\Gam$ \textit{normalizes} $B$
if $hBh^{-1}\subseteq B$ for all $h\in H$.

\begin{defn}[Bi-gyrotransversal]\label{def: Bigyrotrasversal}
A bi-transversal $B$ of subgroups $H_L$ and $H_R$ in a group $\Gam$
is a \textit{bi-gyrotransversal} if
\begin{enumerate}
    \item $H_L$ and $H_R$ normalize $B$, and
    \item $h_\ell h_r = h_r h_\ell$ for all $h_\ell\in H_L$, $h_r\in H_R$.
\end{enumerate}
\end{defn}

\begin{prop}
If $B$ is a bi-gyrotransversal of subgroups $H_L$ and $H_R$ in a
group $\Gam$, then $H_L H_R$ is a subgroup of $\Gam$ with normal
subgroups $H_L$ and $H_R$. If $B$ contains the identity $1$ of
$\Gam$, then $H_L\cap H_R = \set{1}$. In this case, $H_L H_R$ is
isomorphic to the direct product $H_L\times H_R$ as groups.
\end{prop}
\begin{proof}
Since $H_L H_R = H_R H_L$, $H_L H_R$ forms a subgroup of $\Gam$ by
Proposition 14 of \cite[Chapter 3]{DDRF2004AA}. If $g\in H_L H_R$,
then $g = h_\ell h_r$ for some $h_\ell\in H_L$ and $h_r\in H_R$. For
any $h\in H_L$, $h_rh = hh_r$ implies $ghg^{-1} = h_\ell h
h^{-1}_\ell\in H_L$. Hence, $gH_L g^{-1}\subseteq H_L$. This proves
$H_L\unlhd H_L H_R$. Similarly, $H_R\unlhd H_L H_R$.

\par Suppose that $1\in B$ and that $h\in H_L\cap H_R$. The unique
decomposition of $1$, $1 = hh^{-1} = h1h^{-1}$, implies $h = 1$.
Hence, $H_L\cap H_R = \set{1}$. It follows from Theorem 9 of
\cite[Chapter 5]{DDRF2004AA} that $H_L H_R\cong H_L\times H_R$ as
groups.
\end{proof}

\begin{thm}\label{thm: conjugation is an automorphism}
Let $B$ be a bi-gyrotransversal of subgroups $H_L$ and $H_R$ in a
group $\Gam$. If $h\in H_LH_R$, then conjugation by $h$ is an
automorphism of $(B, \odot)$.
\end{thm}
\begin{proof}
Note first that $H_LH_R$ normalizes $B$. In fact, if $h = h_\ell
h_r$ with $h_\ell$ in $H_L$ and $h_r$ in $H_R$, then $hBh^{-1} =
h_\ell(h_rBh_r^{-1})h_\ell^{-1}\subseteq B$ for $H_R$ and $H_L$
normalize $B$.

\par Let $h\in H_LH_R$. Since $H_LH_R$ normalizes $B$,
$\alp_h$ is a bijection from $B$ to itself. Next, we will show that
$(x\odot y)^h = x^h\odot y^h$ for all $x, y\in B$. Employing
\eqref{eqn: Unique decomposition}, we have
$$
(xy)^h = (h_\ell(x, y)(x\odot y)h_r(x, y))^h = h_\ell(x, y)^h(x\odot
y)^hh_r(x, y)^h.
$$
Since $x^h, y^h\in B$, we also have
$$x^hy^h = h_\ell(x^h, y^h)(x^h\odot y^h)h_r(x^h, y^h).$$
Note that $h_\ell(x, y)^h\in H_L$ and $h_r(x, y)^h \in H_R$ because
$H_L$ and $H_R$ are normal in $H_LH_R$. Thus, $(xy)^h = x^hy^h$
implies
$$
h_\ell(x, y)^h = h_\ell(x^h, y^h), \quad (x\odot y)^h = x^h\odot
y^h,\quad\textrm{and}\quad h_r(x, y)^h = h_r(x^h, y^h),
$$
which completes the proof.
\end{proof}

\begin{cor}\label{cor: bigyration is automorphism}
Let $B$ be a bi-gyrotransversal of subgroups $H_L$ and $H_R$ in a
group $\Gam$. Then $\lgyr{a, b}{}$ and $\rgyr{a, b}{}$ are
automorphisms of $(B, \odot)$ for all $a, b\in B$.
\end{cor}
\begin{proof}
This is because $\lgyr{a, b}{} = \alp_{h_r(a, b)}$ and $\rgyr{a,
b}{} = \alp_{h_\ell(a, b)^{-1}}$.
\end{proof}

\par The next theorem provides us with \textit{commuting
relations} between conjugation automorphisms of the bi-transversal
groupoid $(B, \odot)$ and its bi-gyrations.

\begin{thm}\label{thm: commuting relation conjugation and gyration automorphism}
Let $B$ be a bi-gyrotransversal of subgroups $H_L$ and $H_R$ in a
group $\Gam$. The following commuting relations hold.

\newpage

\begin{enumerate}
    \item\label{item: left and right gyration commute} $\lgyr{a, b}{}\circ\rgyr{c, d}{} = \rgyr{c, d}{}\circ\lgyr{a,
    b}{}$ for all $a, b, c, d\in B$.
    \item\label{item: commuting relation of conjugation and left gyration} $\alp_h\circ\lgyr{a, b}{} = \lgyr{\alp_h(a),
    \alp_h(b)}{}\circ\alp_h$ for all $h\in H_LH_R$ and $a, b\in B$.
    \item\label{item: commuting relation of conjugation and right gyration} $\alp_h\circ\rgyr{a, b}{} = \rgyr{\alp_h(a),
    \alp_h(b)}{}\circ\alp_h$ for all $h\in H_LH_R$ and $a, b\in B$.
\end{enumerate}
\end{thm}
\begin{proof}
Item \eqref{item: left and right gyration commute} follows from the
fact that $h_\ell h_r = h_rh_\ell$ for all $h_\ell\in H_L$ and
$h_r\in H_R$ and that $\alp_{gh} =\alp_g\circ\alp_h$ for all $g,
h\in \Gam$.

\par Let $h\in H_LH_R$ and let $a, b\in B$. As in the proof of Theorem
\ref{thm: conjugation is an automorphism}, $h_r(a, b)^h = h_r(a^h,
b^h)$. Hence, $\alp_h\circ\lgyr{a, b}{}\circ\alp_h^{-1} = \lgyr{a^h,
b^h}{}$ and Item \eqref{item: commuting relation of conjugation and
left gyration} follows. Similarly, $h_\ell(a, b)^h = h_\ell(a^h,
b^h)$ implies Item \eqref{item: commuting relation of conjugation
and right gyration}.
\end{proof}

\par As a consequence of Theorem \ref{thm: commuting relation conjugation and gyration
automorphism}, left gyrations are invariant under right gyrations,
and vice versa. In fact, we have the following two theorems.

\begin{thm}\label{thm: invariant of bigyration}
Let $B$ be a bi-gyrotransversal of subgroups $H_L$ and $H_R$ in a
group $\Gam$. If $\rho$ is a finite composition of right gyrations
of $B$, then
\begin{equation}\label{eqn: invariant of left gyration under right}
\lgyr{a, b}{} = \lgyr{\rho(a), \rho(b)}{}
\end{equation}
for all $a, b\in B$. If $\lamb$ is a finite composition of left
gyrations of $B$, then
\begin{equation}\label{eqn: invariant of right gyration under left}
\rgyr{a, b}{} = \rgyr{\lamb(a), \lamb(b)}{}
\end{equation}
for all $a, b\in B$.
\end{thm}
\begin{proof}
By assumption, $\rho = \rgyr{a_1, b_1}{}\circ\rgyr{a_2,
b_2}{}\circ\cdots\circ\rgyr{a_n, b_n}{}$ for some $a_i, b_i\in B$.
Since $\rgyr{a_i, b_i}{} = \alp_{h_\ell(a_i, b_i)^{-1}}$ for all
$i$, it follows that $\rho = \alp_h$, where $h =
h_\ell(a_1,b_1)^{-1}h_\ell(a_2,b_2)^{-1}\cdots h_\ell(a_n,
b_n)^{-1}$. As $\rho = \alp_h$ and $h\in H_L$, Theorem \ref{thm:
commuting relation conjugation and gyration automorphism}
\eqref{item: commuting relation of conjugation and left gyration}
implies $\rho\circ\lgyr{a, b}{} = \lgyr{\rho(a),
\rho(b)}{}\circ\rho$. Since $\rho$ and $\lgyr{a, b}{}$ commute, we
have \eqref{eqn: invariant of left gyration under right}. One
obtains similarly that $\lamb = \alp_h$ for some $h\in H_R$, which
implies \eqref{eqn: invariant of right gyration under left} by
Theorem \ref{thm: commuting relation conjugation and gyration
automorphism} \eqref{item: commuting relation of conjugation and
right gyration}.
\end{proof}

\begin{thm}\label{thm: commuting relation of bigyration}
Let $B$ be a bi-gyrotransversal of subgroups $H_L$ and $H_R$ in a
group $\Gam$. If $\rho$ is a finite composition of right gyrations
of $B$, then
\begin{equation}\label{eqn: commuting relation of rho and right gyration}
\rho\circ\rgyr{a, b}{} = \rgyr{\rho(a), \rho(b)}{}\circ\rho
\end{equation}
for all $a, b\in B$. If $\lamb$ is a finite composition of left
gyrations of $B$, then
\begin{equation}\label{eqn: commuting relation of lambda and left gyration}
\lamb\circ\lgyr{a, b}{} = \lgyr{\lamb(a), \lamb(b)}{}\circ\lamb
\end{equation}
for all $a, b\in B$.
\end{thm}
\begin{proof}
As in the proof of Theorem \ref{thm: invariant of bigyration}, $\rho
= \alp_h$ for some $h\in H_L$. Hence, \eqref{eqn: commuting relation
of rho and right gyration} is an application of Theorem \ref{thm:
commuting relation conjugation and gyration automorphism}
\eqref{item: commuting relation of conjugation and right gyration}.
Similarly, \eqref{eqn: commuting relation of lambda and left
gyration} is an application of Theorem \ref{thm: commuting relation
conjugation and gyration automorphism} \eqref{item: commuting
relation of conjugation and left gyration}.
\end{proof}

\par The associativity of $\Gam$ is reflected in
its bi-gyrotransversal decomposition $\Gam = H_LBH_R$, as shown in
the following theorem.

\begin{thm}\label{thm: bigyroassociative law}
Let $B$ be a bi-gyrotransversal of subgroups $H_L$ and $H_R$ in a
group $\Gam$. For all $a, b, c\in B$,
$$(a\odot b)\odot\lgyr{a, b}{c} = \rgyr{b, c}{a}\odot(b\odot c).$$
\end{thm}
\begin{proof}
Let $a, b, c\in B$. Set $a_r = \rgyr{b,c}{a}$ and $c_l = \lgyr{a,
b}{c}$. Then $a_r\in B$ and $c_l\in B$. By employing \eqref{eqn:
Unique decomposition},
\begin{align}
a(bc) &= a(h_\ell(b, c)(b\odot c)h_r(b, c))\notag\\
{} &= h_\ell(b, c)(h_\ell(b, c)^{-1}ah_\ell(b, c))(b\odot c)h_r(b,
c)\notag\\
{} &= h_\ell(b, c)a_r(b\odot c)h_r(b, c)\notag\\
{} &= [h_\ell(b, c)h_\ell(a_r, b\odot c)][a_r\odot(b\odot
c)][h_r(a_r, b\odot c)h_r(b, c)]\notag
\end{align}
and, similarly, $(ab)c = [h_\ell(a, b)h_\ell(a\odot b, c_l)][(a\odot
b)\odot c_l][h_r(a\odot b, c_l)h_r(a, b)]$. Since $a(bc) = (ab)c$,
it follows that $(a\odot b)\odot c_l = a_r\odot(b\odot c)$, which
was to be proved.
\end{proof}

\begin{prop}\label{prop: relation of bigyrarion I}
Let $B$ be a bi-gyrotransversal of subgroups $H_L$ and $H_R$ in a
group $\Gam$. For all $a, b, c\in B$,
\begin{enumerate}
    \item\label{item: relation of right gyration I} $\rgyr{\rgyr{b, c}{a}, b\odot c}{}\circ\rgyr{b, c}{} = \rgyr{a\odot b, \lgyr{a, b}{c}}{}\circ\rgyr{a,
    b}{}$, and
    \item\label{item: relation of left gyration I} $\lgyr{a\odot b, \lgyr{a, b}{c}}{}\circ\lgyr{a, b}{} = \lgyr{\rgyr{b, c}{a}, b\odot c}{}\circ\lgyr{b,
    c}{}$.
\end{enumerate}
\end{prop}
\begin{proof}
As we have computed in the proof of Theorem \ref{thm:
bigyroassociative law},
$$h_\ell(b, c)h_\ell(a_r, b\odot c) = h_\ell(a, b)h_\ell(a\odot b, c_l),$$
where $a_r = \rgyr{b, c}{a}$ and $c_l = \lgyr{a, b}{c}$. Thus, Item
\eqref{item: relation of right gyration I} is obtained. Similarly,
$h_r(a_r, b\odot c)h_r(b, c) = h_r(a\odot b, c_l)h_r(a, b)$ gives
Item \eqref{item: relation of left gyration I}.
\end{proof}

\subsection*{Twisted subgroups}\label{subsec: twisted subgroup}
\par Twisted subgroups abound in group theory, gyrogroup theory, and loop
theory, as evidenced, for instance, from \cite{MA1998NSF, TF2003SNS,
TFAU2000IDG, RLAY2013TRG, TFMKJP2006OTS, MA2005OBL, MAMKJP2005FBL}.
Here, we demonstrate that a bi-gyrotransversal decomposition $\Gam =
H_LBH_R$ in which $B$ is a twisted subgroup gives rise to a highly
structured type of bi-gyrogroupoids and, eventually, a bi-gyrogroup.
We follow Aschbacher for the definition of a twisted subgroup.

\begin{defn}[Twisted subgroup]\label{def: twisted subgroup}
A subset $B$ of a group $\Gam$ is a \textit{twisted subgroup} of
$\Gam$ if the following conditions hold:
\begin{enumerate}
    \item $1\in B$, $1$ being the identity of $\Gam$;
    \item if $b\in B$, then $b^{-1}\in B$;
    \item if $a, b\in B$, then $aba\in B$.
\end{enumerate}
\end{defn}

\newpage

\begin{thm}\label{thm: Properties of bigyrotransversal and twisted subgroup I}
Let $B$ be a bi-gyrotransversal of subgroups $H_L$ and $H_R$ in a
group $\Gam$. If $B$ is a twisted subgroup of $\Gam$, then the
following properties are true for all $a, b\in B$.
\begin{enumerate}
    \item\label{item: identity in bigyrotransversal} $1\odot b = b\odot 1 = b$.
    \item\label{item: inverse in bigyrotransversal} $b^{-1}\in B$ and $b^{-1}\odot b = b\odot b^{-1} = 1$.
    \item\label{item: bigyration of 1 and a} $\lgyr{1, b}{} = \lgyr{b, 1}{} = \rgyr{1, b}{} = \rgyr{b, 1}{} =
    \id{B}$.
    \item\label{item: left and right gyration of b and b inverse}
    $\lgyr{b^{-1}, b}{} = \lgyr{b, b^{-1}}{} = \rgyr{b^{-1}, b}{} = \rgyr{b, b^{-1}}{} =
    \id{B}$.
    \item\label{item: inverse of left and right bigyration} $\ilgyr{a, b}{} = \lgyr{b^{-1}, a^{-1}}{}$ and $\irgyr{a, b}{} = \rgyr{b^{-1},
    a^{-1}}{}$.
    \item\label{item: (a . b)^-1} $(a\odot b)^{-1} = (\lgyr{a, b}{}\circ\rgyr{a, b}{})(b^{-1}\odot
    a^{-1})$.
\end{enumerate}
\end{thm}
\begin{proof}
\eqref{item: identity in bigyrotransversal} As $b = 1b = h_\ell(1,
b)(1\odot b)h_r(1, b)$, we have $h_\ell(1, b) = 1$, $1\odot b = b$,
and $h_r(1, b) = 1$. Similarly, $b = b1$ implies $b\odot 1 = b$.

\par \eqref{item: inverse in bigyrotransversal} Let $b\in B$. Since $B$ is a
twisted subgroup, $b^{-1}\in B$. Further,
$$1 = b^{-1}b = h_\ell(b^{-1}, b)(b^{-1}\odot b)h_r(b^{-1}, b)$$
implies $h_\ell(b^{-1}, b) = 1$, $b^{-1}\odot b = 1$, and
$h_r(b^{-1}, b) = 1$. Similarly, $bb^{-1} = 1$ implies $b\odot
b^{-1} = 1$.

\par \eqref{item: bigyration of 1 and a} We have $h_\ell(1, b) = h_\ell(b, 1) = h_r(1, b)
= h_r(b, 1) = 1$, as computed in Item \eqref{item: identity in
bigyrotransversal}. Hence, Item \eqref{item: bigyration of 1 and a}
follows.

\par \eqref{item: left and right gyration of b and b inverse} We have $h_\ell(b^{-1}, b) = h_\ell(b, b^{-1}) =
h_r(b^{-1}, b) = h_r(b, b^{-1}) = 1$, as computed in Item
\eqref{item: inverse in bigyrotransversal}. Hence, Item \eqref{item:
left and right gyration of b and b inverse} follows.

\par \eqref{item: inverse of left and right bigyration} Let $a, b\in
B$. Then $a^{-1}, b^{-1}\in B$. On the one hand, we have
$$(ab)^{-1} = (h_\ell(a, b)(a\odot b)h_r(a, b))^{-1} = h_r(a, b)^{-1}(a\odot b)^{-1}h_\ell(a, b)^{-1},$$
and on the other hand we have $b^{-1}a^{-1} = h_\ell(b^{-1},
a^{-1})(b^{-1}\odot a^{-1})h_r(b^{-1}, a^{-1})$. Since $(ab)^{-1} =
b^{-1}a^{-1}$, it follows that
\begin{eqnarray}\label{eqn: (a . b)^-1}
\begin{split}
(a\odot b)^{-1} &= h_r(a, b)h_\ell(b^{-1}, a^{-1})(b^{-1}\odot
a^{-1})h_r(b^{-1}, a^{-1})h_\ell(a, b)\\
{} &= h_\ell(b^{-1}, a^{-1})h_r(a, b)(b^{-1}\odot a^{-1})h_\ell(a,
b)h_r(b^{-1}, a^{-1})\\
{} &= h_\ell(b^{-1}, a^{-1})h_\ell(a, b)\tilde{b}h_r(a,
b)h_r(b^{-1}, a^{-1}),
\end{split}
\end{eqnarray}
where $\tilde{b} = \lgyr{a, b}{(\rgyr{a, b}{(b^{-1}\odot
a^{-1})})}$. Because $(a\odot b)^{-1}$ and $\tilde{b}$ belong to
$B$, we have from the extreme sides of \eqref{eqn: (a . b)^-1} that
$$h_r(a,
b)h_r(b^{-1}, a^{-1}) = 1\quad\textrm{and}\quad h_\ell(b^{-1},
a^{-1})h_\ell(a, b) = 1.$$ Hence, $h_r(a, b)^{-1} = h_r(b^{-1},
a^{-1})$, which implies $\ilgyr{a, b}{} = \lgyr{b^{-1}, a^{-1}}{}$.
\mbox{Similarly}, $h_\ell(a, b) = h_\ell(b^{-1}, a^{-1})^{-1}$
implies $\irgyr{a, b}{} = \rgyr{b^{-1}, a^{-1}}{}$.

\par \eqref{item: (a . b)^-1} As in Item \eqref{item: inverse of left and right
bigyration}, $(a\odot b)^{-1} = \tilde{b} = \lgyr{a, b}{(\rgyr{a,
b}{(b^{-1}\odot a^{-1})})}$.
\end{proof}

\begin{rem} Note that we do not invoke the third defining property of a
twisted subgroup in proving Theorem \ref{thm: Properties of
bigyrotransversal and twisted subgroup I}.
\end{rem}

\par At this point, we have shown that any bi-gyrotransversal
decomposition $\Gam = H_LBH_R$ in which $B$ is a twisted subgroup of
$\Gam$ gives the bi-transversal groupoid $B$ that satisfies all the
axioms of a bi-gyrogroupoid except for (BG4). In order to complete
this, we have to impose additional conditions on the left and right
transversal maps, as the following lemma indicates.

\begin{lem}\label{lem: h(a, b)^{-1} = h (b, a) and inverse bigyration}
If $B$ is a bi-transversal of subgroups $H_L$ and $H_R$ in a group
$\Gam$ such that $h_\ell(a, b)^{-1} = h_\ell(b, a)$ and $h_r(a,
b)^{-1} = h_r(b, a)$ for all $a, b\in B$, then
$$\ilgyr{a, b}{} = \lgyr{b, a}{}\quad\textrm{and}\quad\irgyr{a, b}{} = \rgyr{b, a}{}$$
for all $a, b\in B$.
\end{lem}
\begin{proof}
Note first that $\alp_h^{-1} = \alp_{h^{-1}}$ for all $h\in\Gam$.
From this we have $\lgyr{b, a}{} = \alp_{h_r(b, a)} = \alp_{h_r(a,
b)^{-1}} = \alp^{-1}_{h_r(a, b)} = \ilgyr{a, b}{}$. One can prove in
a similar way that $\irgyr{a, b}{} = \rgyr{b, a}{}$.
\end{proof}

\begin{thm}\label{thm: like left loop properties}
Let $B$ be a bi-gyrotransversal of subgroups $H_L$ and $H_R$ in a
group $\Gam$. If $B$ is a twisted subgroup of $\Gam$ such that
$h_\ell(a, b)^{-1} = h_\ell(b, a)$ and $h_r(a, b)^{-1} = h_r(b, a)$
for all $a, b\in B$, then the following relations hold for all $a,
b\in B$:
\begin{enumerate}
    \item\label{item: like left loop property, right gyration} $\rgyr{a, b}{} = \rgyr{\lgyr{a, b}{a}, a\odot b}{}$;
    \item\label{item: like left loop property, left gyration} $\lgyr{a, b}{} = \lgyr{\lgyr{a, b}{a}, a\odot b}{}$;
    \item\label{item: like right loop property, right gyration} $\rgyr{a, b}{} = \rgyr{\rgyr{b, a}{a}, b\odot a}{}$;
    \item\label{item: like right loop property, left gyration} $\lgyr{a, b}{} = \lgyr{\rgyr{b, a}{a}, b\odot a}{}$.
\end{enumerate}
\end{thm}
\begin{proof}
Let $a, b\in B$. Set $a_l = \lgyr{a, b}{a}$. Employing \eqref{eqn:
Unique decomposition}, we obtain
\begin{eqnarray}\label{eqn: aba decomposition}
\begin{split}
(ab)a &= (h_\ell(a, b)(a\odot b)h_r(a, b))a\\
{} &= h_\ell(a, b)(a\odot b)a_lh_r(a, b)\\
{} &= [h_\ell(a, b)h_\ell(a\odot b, a_l)][(a\odot b)\odot
a_l][h_r(a\odot b, a_l)h_r(a, b)].
\end{split}
\end{eqnarray}
Since $(ab)a\in B$, the extreme sides of \eqref{eqn: aba
decomposition} imply
\begin{equation}\label{eqn: In proof of properties of bigyration}
h_\ell(a, b)h_\ell(a\odot b, a_l) = 1\quad\textrm{and}\quad
h_r(a\odot b, a_l)h_r(a, b) = 1.
\end{equation}
The first equation of \eqref{eqn: In proof of properties of
bigyration} implies $h_\ell(a\odot b, \lgyr{a, b}{a}) = h_\ell(a,
b)^{-1}$. Hence, $\irgyr{a\odot b, \lgyr{a, b}{a}}{} = \rgyr{a,
b}{}$. From Lemma \ref{lem: h(a, b)^{-1} = h (b, a) and inverse
bigyration}, we have $$\rgyr{a, b}{} = \rgyr{\lgyr{a, b}{a}, a\odot
b}{}.$$ The second equation of \eqref{eqn: In proof of properties of
bigyration} implies $h_r(a, b) = h_r(a\odot b, \lgyr{a,
b}{a})^{-1}$. Hence, $\lgyr{a, b}{} = \lgyr{\lgyr{a, b}{a}, a\odot
b}{}$. This proves Items \eqref{item: like left loop property, right
gyration} and \eqref{item: like left loop property, left gyration}.
Items \eqref{item: like right loop property, right gyration} and
\eqref{item: like right loop property, left gyration} can be proved
in a similar way by computing the \mbox{product} $a(ba)$.
\end{proof}

\begin{thm}\label{thm: even bigyration}
Let $B$ be a bi-gyrotransversal of subgroups $H_L$ and $H_R$ in a
group $\Gam$. If $B$ is a twisted subgroup of $\Gam$ such that
$h_\ell(a, b)^{-1} = h_\ell(b, a)$ and $h_r(a, b)^{-1} = h_r(b, a)$
for all $a, b\in B$, then left and right gyrations of $B$ are
\textit{even} in the sense that
$$\lgyr{a^{-1}, b^{-1}}{} = \lgyr{a, b}{}\quad\textrm{and}\quad\rgyr{a^{-1}, b^{-1}}{} = \rgyr{a, b}{}$$
for all $a, b\in B$.
\end{thm}
\begin{proof}
This theorem follows directly from Theorem \ref{thm: Properties of
bigyrotransversal and twisted subgroup I} \eqref{item: inverse of
left and right bigyration} and Lemma \ref{lem: h(a, b)^{-1} = h (b,
a) and inverse bigyration}.
\end{proof}

\section{Bi-gyrodecomposition and bi-gyrogroups}\label{sec: bigyrodecomposition}
\par Taking the key features of bi-gyrotransversal decomposition
of a group given in Section \ref{sec: bitransversal}, we formulate
the definition of bi-gyrodecomposition and show that any
bi-gyrodecomposition leads to a bi-gyrogroup, which in turn is a
gyrogroup. Most of the results in Section \ref{sec: bitransversal}
are directly translated into results in this section with
appropriate modifications.

\begin{defn}[Bi-gyrodecomposition]\label{def: Bi-gyrodecomposition}
Let $\Gam$ be a group, let $B$ be a subset of $\Gam$, and let $H_L$
and $H_R$ be subgroups of $\Gam$. A decomposition $\Gam = H_LBH_R$
is a \textit{bi-gyrodecomposition} if
\begin{enumerate}
    \item $B$ is a bi-gyrotransversal of $H_L$ and $H_R$ in $\Gam$;
    \item $B$ is a twisted subgroup of $\Gam$; and
    \item $h_\ell(a, b)^{-1} = h_\ell(b, a)$ and $h_r(a, b)^{-1} = h_r(b,
    a)$ for all $a, b\in B$,
\end{enumerate}
where $h_\ell$ and $h_r$ are the bi-transversal maps given below
Definition \ref{def: Bitrasversal}.
\end{defn}

\begin{thm}\label{thm: bigyrodecomposition is a bigyrogrouppoid}
If $\Gam = H_LBH_R$ is a bi-gyrodecomposition, then $B$ equipped
with the bi-transversal operation forms a bi-gyrogroupoid.
\end{thm}
\begin{proof}
Axiom (BG1) holds by Theorem \ref{thm: Properties of
bigyrotransversal and twisted subgroup I} \eqref{item: identity in
bigyrotransversal}, where the identity $1$ of $\Gam$ acts as the
identity of $B$. Axiom (BG2) holds by Theorem \ref{thm: Properties
of bigyrotransversal and twisted subgroup I} \eqref{item: inverse in
bigyrotransversal}, where $b^{-1}$ acts as a left inverse of $b\in
B$ with respect to the bi-transversal operation. Axiom (BG3) holds
by Corollary \ref{cor: bigyration is automorphism} and Theorem
\ref{thm: bigyroassociative law}. Axiom (BG4) holds by Theorem
\ref{thm: like left loop properties}. Axiom (BG5) holds by Theorem
\ref{thm: Properties of bigyrotransversal and twisted subgroup I}
\eqref{item: bigyration of 1 and a}.
\end{proof}

\par It is shown in Section \ref{sec: bitransversal} that
any bi-transversal decomposition $\Gam = H_LBH_R$ gives rise to a
bi-transversal groupoid $(B, \odot)$. Theorem \ref{thm:
bigyrodecomposition is a bigyrogrouppoid} asserts that in the
special case when the decomposition is a bi-gyrodecomposition, the
bi-transversal groupoid $(B, \odot)$ becomes the bi-gyrogroupoid
$(B, \oplus_b)$ described in Definition \ref{def: bigyrogroupoid}.
Hence, in particular, the binary operations $\oplus_b$ and $\odot$
share the same algebraic properties. Further, the identity of the
bi-gyrogroupoid $B$ coincides with the group identity of $\Gam$ and
$\ominus_b b = b^{-1}$ for all $b\in B$.

\begin{thm}[Bi-gyration invariant relation]\label{thm: bigyration invariant relation}
Let $\Gam = H_L B H_R$ be a bi-gyrodecomposition. If $\rho$ is a
finite composition of right gyrations of $B$, then
\begin{equation}\label{eqn: left gyration invariant relation}
\lgyr{a, b}{} = \lgyr{\rho(a), \rho(b)}{}
\end{equation}
for all $a, b\in B$. If $\lamb$ is a finite composition of left
gyrations of $B$, then
\begin{equation}\label{eqn: right gyration invariant relation}
\rgyr{a, b}{} = \rgyr{\lamb(a), \lamb(b)}{}
\end{equation}
for all $a, b\in B$.
\end{thm}
\begin{proof}
The theorem follows immediately from Theorem \ref{thm: invariant of
bigyration}.
\end{proof}

\begin{thm}[Bi-gyration commuting relation]\label{thm: bigyrodecomposition, commuting relation of bigyration}
Let $\Gam = H_L B H_R$ be a bi-gyrodecomposition. If $\rho$ is a
finite composition of right gyrations of $B$, then
\begin{equation}\label{eqn: bigyrodecomposition, commuting relation of rho and right gyration}
\rho\circ\rgyr{a, b}{} = \rgyr{\rho(a), \rho(b)}{}\circ\rho
\end{equation}
for all $a, b\in B$. If $\lamb$ is a finite composition of left
gyrations of $B$, then
\begin{equation}\label{eqn: bigyrodecomposition, commuting relation of lambda and left gyration}
\lamb\circ\lgyr{a, b}{} = \lgyr{\lamb(a), \lamb(b)}{}\circ\lamb
\end{equation}
for all $a, b\in B$.
\end{thm}
\begin{proof}
The theorem follows immediately from Theorem \ref{thm: commuting
relation of bigyration}.
\end{proof}

\begin{thm}[Trivial bi-gyration]
If $\Gam = H_L B H_R$ is a bi-gyrodecomposition, then for all $a\in
B$,
\begin{eqnarray}
\begin{split}
\lgyr{0, a}{} &= \lgyr{a, 0}{} &= \id{B}\\
\lgyr{a, \ominus_b a}{} &= \lgyr{\ominus_b a, a}{} &= \id{B}\\
\rgyr{0, a}{} &= \rgyr{a, 0}{} &= \id{B}\\
\rgyr{a, \ominus_b a}{} &= \rgyr{\ominus_b a, a}{} &= \id{B}\\
\lgyr{a, a}{} &= \rgyr{a, a}{} &= \id{B}.
\end{split}
\end{eqnarray}
\end{thm}
\begin{proof}
The theorem follows from Theorem \ref{thm: basic properties of
bigyrogroup I} \eqref{item: bigyration generated by a,a} and Theorem
\ref{thm: Properties of bigyrotransversal and twisted subgroup I}
\eqref{item: bigyration of 1 and a}--\eqref{item: left and right
gyration of b and b inverse}.
\end{proof}

\begin{thm}[Bi-gyration inversion law]\label{thm: bigyrodecomposition inversion Law}
If $\Gam = H_LBH_R$ is a \hfill\\ bi-gyrodecomposition, then
$$\ilgyr{a, b}{} = \lgyr{b, a}{}\quad\textrm{and}\quad\irgyr{a, b}{} = \rgyr{b, a}{}$$
for all $a, b\in B$.
\end{thm}
\begin{proof}
The theorem follows immediately from Lemma \ref{lem: h(a, b)^{-1} =
h (b, a) and inverse bigyration}.
\end{proof}

\begin{thm}[Even bi-gyration]\label{thm: bigyrodecomposition is even}
If $\Gam = H_LBH_R$ is a bi-gyrodecomposition, then left and right
gyrations of $B$ are even:
$$\lgyr{\ominus_b a, \ominus_b b}{} = \lgyr{a, b}{}\quad\textrm{and}\quad\rgyr{\ominus_b a, \ominus_b b}{} = \rgyr{a, b}{}$$
for all $a, b\in B$.
\end{thm}
\begin{proof}
The theorem follows immediately from Theorem \ref{thm: even
bigyration}.
\end{proof}

\begin{thm}[Left and right cancellation laws]\label{thm: bigyrodecomposition cancellation law}
If $\Gam = H_LBH_R$ is a bi-gyrodecomposition, then $B$ satisfies
the left cancellation law
\begin{equation}\label{eqn: bigyrodecomposition left cancellation law}
\ominus_b\rgyr{a, b}{a}\oplus_b (a\oplus_b b) = b
\end{equation}
and the right cancellation law
\begin{equation}\label{eqn: bigyrodecomposition right cancellation law}
(a\oplus_b b)\ominus_b\lgyr{a, b}{b} = a
\end{equation}
for all $a, b\in B$.
\end{thm}
\begin{proof}
The theorem follows immediately from Theorem \ref{thm: left-right
cancellation law}.
\end{proof}

\begin{thm}[Left and right bi-gyroassociative laws]\label{thm: bigyrodecomposition bigyroassociative law}
If $\Gam = H_LBH_R$ is a bi-gyrodecomposition, then $B$ satisfies
the left bi-gyroassociative law
\begin{equation}\label{eqn: bigyrodecomposition left bigyroassociative law}
a\oplus_b (b\oplus_b c) = (\rgyr{c, b}{a}\oplus_b
b)\oplus_b\lgyr{\rgyr{c, b}{a}, b}{c}
\end{equation}
and the right bi-gyroassociative law
\begin{equation}\label{eqn: bigyrodecomposition right bigyroassociative law}
(a\oplus_b b)\oplus_b c = \rgyr{b, \lgyr{b, a}{c}}{a}\oplus_b
(b\oplus_b \lgyr{b, a}{c})
\end{equation}
for all $a, b, c\in B$.
\end{thm}
\begin{proof}
The theorem follows from Theorems \ref{thm: left and right
bi-gyroassociative law} and \ref{thm: bigyrodecomposition inversion
Law}.
\end{proof}

\begin{thm}[Left gyration reduction property]\label{thm: left gyration reduction property}
If $\Gam = H_LBH_R$ is a bi-gyrodecomposition, then
\begin{equation}\label{eqn: left gyration reduction property I}
\lgyr{a, b}{} = \lgyr{\rgyr{b, a}{a}, b\oplus_b a}{}
\end{equation}
and
\begin{equation}\label{eqn: left gyration reduction property II}
\lgyr{a, b}{} = \lgyr{a\oplus_b b, \rgyr{a, b}{b}}{}
\end{equation}
for all $a, b\in B$.
\end{thm}
\begin{proof}
Identity \eqref{eqn: left gyration reduction property I} follows
from Theorem \ref{thm: like left loop properties} \eqref{item: like
right loop property, left gyration}. Identity \eqref{eqn: left
gyration reduction property II} is obtained from \eqref{eqn: left
gyration reduction property I} by \mbox{applying} the bi-gyration
inversion law (Theorem \ref{thm: bigyrodecomposition inversion Law})
followed by interchanging $a$ and $b$.
\end{proof}

\begin{thm}[Right gyration reduction property]\label{thm: right gyration reduction property}
If $\Gam = H_LBH_R$ is a bi-gyrodecomposition, then
\begin{equation}\label{eqn: right gyration reduction property I}
\rgyr{a, b}{} = \rgyr{\lgyr{a, b}{a}, a\oplus_b b}{}
\end{equation}
and
\begin{equation}\label{eqn: right gyration reduction property II}
\rgyr{a, b}{} = \rgyr{b\oplus_b a, \lgyr{b, a}{b}}{}
\end{equation}
for all $a, b\in B$.
\end{thm}
\begin{proof}
Identity \eqref{eqn: right gyration reduction property I} follows
from Theorem \ref{thm: like left loop properties} \eqref{item: like
left loop property, right gyration}. Identity \eqref{eqn: right
gyration reduction property II} is obtained from \eqref{eqn: right
gyration reduction property I} by \mbox{applying} the bi-gyration
inversion law followed by \mbox{interchanging} $a$ and $b$.
\end{proof}

\begin{thm}[Bi-gyration reduction property]\label{thm: bigyration reduction property}
If $\Gam = H_LBH_R$ is a bi-gyrodecomposition, then
\begin{equation}\label{eqn: bigyration reduction property I}
\lgyr{a, b}{} = \lgyr{\lgyr{a, b}{a}, a\oplus_b b}{}
\end{equation}
and
\begin{equation}\label{eqn: bigyration reduction property II}
\rgyr{a, b}{} = \rgyr{a\oplus_b b, \rgyr{a, b}{b}}{}
\end{equation}
for all $a, b\in B$.
\end{thm}
\begin{proof}
Identity \eqref{eqn: bigyration reduction property I} follows from
Theorem \ref{thm: like left loop properties} \eqref{item: like left
loop property, left gyration}. Identity \eqref{eqn: bigyration
reduction property II} is obtained from Theorem \ref{thm: like left
loop properties} \eqref{item: like right loop property, right
gyration} by applying the bi-gyration inversion law \mbox{followed}
by interchanging $a$ and $b$.
\end{proof}

\begin{thm}[Left and right gyration reduction properties]\label{thm: left and right gyration reduction property}
If $\Gam = H_LBH_R$ is a bi-gyrodecomposition, then
\begin{eqnarray}
\begin{split}\label{eqn: left and right gyration reduction property I}
\rgyr{a, b}{} &= \rgyr{\ominus_b\lgyr{a, b}{b}, a\oplus_b b}{}\\
\lgyr{a, b}{} &= \lgyr{\ominus_b\lgyr{a, b}{b}, a\oplus_b b}{}
\end{split}
\end{eqnarray}
and
\begin{eqnarray}
\begin{split}\label{eqn: left and right gyration reduction property II}
\rgyr{a, b}{} &= \rgyr{a\oplus_b b, \ominus_b\rgyr{a, b}{a}}{}\\
\lgyr{a, b}{} &= \lgyr{a\oplus_b b, \ominus_b\rgyr{a, b}{a}}{}
\end{split}
\end{eqnarray}
for all $a, b\in B$.
\end{thm}
\begin{proof}
Setting $c = \ominus_b b$ in Proposition \ref{prop: relation of
bigyrarion I} \eqref{item: relation of right gyration
I}--\eqref{item: relation of left gyration I} followed by using the
bi-gyration inversion law gives \eqref{eqn: left and right gyration
reduction property I}. Setting $a = \ominus_b b$ in the same
proposition followed by using the bi-gyration inversion law gives
\begin{align}
\rgyr{b, c}{} &= \rgyr{b\oplus_b c, \ominus_b\rgyr{b, c}{b}}{}\notag\\
\lgyr{b, c}{} &= \lgyr{b\oplus_b c, \ominus_b\rgyr{b,c}{b}}{}.\notag
\end{align}
Replacing $b$ by $a$ and $c$ by $b$, we obtain \eqref{eqn: left and
right gyration reduction property II}.
\end{proof}

\begin{thm}[Left and right gyration reduction properties]\label{thm: reduction property with -a}
If $\Gam = H_LBH_R$ is a bi-gyrodecomposition, then
\begin{eqnarray}\label{eqn: reduction property with -a}
\begin{split}
\lgyr{a, b}{} &= \lgyr{\rgyr{b, a}{(a\oplus_b b)}, \ominus_b a}{}\\
\rgyr{a, b}{} &= \rgyr{\rgyr{b, a}{(a\oplus_b b)}, \ominus_b a}{}
\end{split}
\end{eqnarray}
for all $a, b\in B$.
\end{thm}
\begin{proof}
From the second equation of \eqref{eqn: left and right gyration
reduction property II}, we have
$$\lgyr{a, b}{} = \lgyr{a\oplus_b b, \ominus_b\rgyr{a, b}{a}}{}.$$
Applying Theorem \ref{thm: bigyration invariant relation} to the
previous equation with $\rho = \rgyr{b, a}{}$ gives
\begin{align}
\lgyr{a, b}{} &= \lgyr{a\oplus_b b, \ominus_b\rgyr{a, b}{a}}{}\notag\\
{} &= \lgyr{\rgyr{b, a}{(a\oplus_b b)}, \rgyr{b,a}{(\ominus_b\rgyr{a, b}{a})}}{}\notag\\
{} &= \lgyr{\rgyr{b, a}{(a\oplus_b b)}, \ominus_b a}{}.\notag
\end{align}
We obtain the last equation since $\rgyr{b, a}{} = \irgyr{a, b}{}$.
\par Similarly, the first equation of \eqref{eqn: left and right gyration
reduction property II} and Identity \eqref{eqn: bigyrodecomposition,
commuting relation of rho and right gyration} together imply
\begin{eqnarray}\label{eqn: in proof left right gyration reduction property}
\begin{split}
\id{B} &= \irgyr{a, b}{}\circ\rgyr{a\oplus_b, \ominus_b\rgyr{a,
b}{a}}{}\\
{} &= \rgyr{b, a}{}\circ\rgyr{a\oplus_b b, \ominus_b\rgyr{a,
b}{a}}{}\\
{} &= \rgyr{\rgyr{b, a}{(a\oplus_b b)}, \rgyr{b,
a}{(\ominus_b\rgyr{a,
b}{a})}}{}\circ\rgyr{b, a}{}\\
{} &= \rgyr{\rgyr{b, a}{(a\oplus_b b)}, \ominus_b a}{}\circ\rgyr{b,
a}{}.
\end{split}
\end{eqnarray}
The extreme sides of \eqref{eqn: in proof left right gyration
reduction property} imply $\rgyr{a, b}{} = \rgyr{\rgyr{b,
a}{(a\oplus_b b)}, \ominus_b a}{}$.
\end{proof}

\subsection*{Bi-gyrogroups}\label{subsec: bigyrogroup}
\par We are now in a position to present the formal definition of a
bi-gyrogroup.

\begin{defn}[Bi-gyrogroup]\label{def: bi-gyrogroup operation}
Let $\Gam = H_LBH_R$ be a bi-gyrodecomposition. The
\textit{bi-gyrogroup operation} $\oplus$ in $B$ is defined by
\begin{equation}\label{eqn: bigyrogroup operation}
a\oplus b = \rgyr{b, a}{(a\oplus_b b)},\quad a, b\in B.
\end{equation}
Here, $\oplus_b$ is the bi-transversal operation induced by the
decomposition $\Gam = H_LBH_R$. The groupoid $(B, \oplus)$
consisting of the set $B$ with the bi-gyrogroup operation $\oplus$
is called a \textit{bi-gyrogroup}.
\end{defn}

\par Throughout the remainder of this section,
we assume that $\Gam = H_LBH_R$ is a bi-gyrodecomposition and let
$(B, \oplus)$ be the corresponding bi-gyrogroup.

\begin{prop}\label{prop: identity and inverse in bigyrogroup}
The unique two-sided identity element of $(B, \oplus)$ is $0$. For
each $a\in B$, $\ominus_b a$ is the unique two-sided inverse of $a$
in $(B,\oplus)$.
\end{prop}
\begin{proof}
Let $a\in B$. Since $\rgyr{a, 0}{} = \rgyr{0, a}{} = \id{B}$, we
have
$$a\oplus 0 = \rgyr{0, a}{(a\oplus_b 0)} = a = \rgyr{a, 0}{(0\oplus_b a)} = (0\oplus a).$$
Hence, $0$ is a two-sided identity of $(B,\oplus)$. The uniqueness
of $0$ follows, as in the proof of Lemma \ref{lem: Uniqueness of 0
in a bigyrogroup}.
\par Since $\rgyr{a, \ominus_b a}{} = \rgyr{\ominus_b a, a}{} = \id{B}$,
we have
$$a\oplus(\ominus_b a) = \rgyr{\ominus_b a, a}{(a\ominus_b a)} = 0 =
\rgyr{a, \ominus_b a}{(\ominus_b a\oplus_b a)} = (\ominus_b a)\oplus
a.$$ Hence, $\ominus_b a$ acts as a two-sided inverse of $a$ with
respect to $\oplus$. Suppose that $b$ is a two-sided inverse of $a$
with respect to $\oplus$. Then $0 = a\oplus b = \rgyr{b,
a}{(a\oplus_b b)}$, which implies $a\oplus_b b = 0$. Similarly,
$b\oplus a = 0$ implies $b\oplus_b a = 0$. This proves that $b$ is a
two-sided inverse of $a$ with respect to $\oplus_b$. Hence, $b =
\ominus_b a$ by Theorem \ref{thm: unique two-sided inverse}.
\end{proof}

\par Following Proposition \ref{prop: identity and inverse in
bigyrogroup}, if $a$ is an element of $B$, then the unique two-sided
inverse of $a$ with respect to $\oplus$ will be denoted by $\ominus
a$. Further, $$\ominus a = \ominus_b a$$ for all $a\in B$. We also
write $a\ominus b$ instead of $a\oplus(\ominus b)$. The following
theorem asserts that left and right gyrations of the bi-transversal
groupoid $(B, \oplus_b)$ ascend to automorphisms of the bi-gyrogroup
$(B, \oplus)$.

\begin{thm}\label{thm: automorphism of bigyrogroup operation}
If $\lamb$ is a finite composition of left gyrations of $(B,
\oplus_b)$, then
\begin{equation}\label{eqn: composite of left bigyration preserve +}
\lamb(a\oplus b) = \lamb(a)\oplus\lamb(b)
\end{equation}
for all $a, b\in B$. If $\rho$ is a finite composition of right
gyrations of $(B,\oplus_b)$, then
\begin{equation}\label{eqn: composite of right bigyration preserve +}
\rho(a\oplus b) = \rho(a)\oplus\rho(b)
\end{equation}
for all $a, b\in B$.
\end{thm}
\begin{proof}
Let $a, b\in B$. By Theorem \ref{thm: commuting relation conjugation
and gyration automorphism} \eqref{item: left and right gyration
commute}, $\lamb$ and $\rgyr{b, a}{}$ commute. Hence,
\begin{align}
\lamb(a\oplus b) &= (\lamb\circ\rgyr{b, a}{})(a\oplus_b b)\notag\\
{} &= (\rgyr{b, a}{}\circ\lamb)(a\oplus_b b)\notag\\
{} &= \rgyr{b, a}{(\lamb(a)\oplus_b\lamb(b))}\notag\\
{} &= \rgyr{\lamb(b), \lamb(a)}{(\lamb(a)\oplus_b\lamb(b))}\notag\\
{} &= \lamb(a)\oplus\lamb(b).\notag
\end{align}
We have the third equation since $\lamb$ is a finite composition of
left gyrations; the forth equation from \eqref{eqn: right gyration
invariant relation}; and the last equation from Definition \ref{def:
bi-gyrogroup operation}. Similarly, \eqref{eqn: composite of right
bigyration preserve +} is obtained from \eqref{eqn:
bigyrodecomposition, commuting relation of rho and right gyration}.
\end{proof}

\begin{lem}\label{lem: bigyration relation in bigyrogroup}
In the bi-gyrogroup $B$,
$$ \rgyr{c, a\oplus b}{}\circ\rgyr{b, a}{} = \rgyr{b\oplus c,
a}{}\circ\rgyr{c, b}{}$$ for all $a, b, c\in B$.
\end{lem}
\begin{proof}
By Theorem \ref{thm: bigyrodecomposition inversion Law} and
Proposition \ref{prop: relation of bigyrarion I} \eqref{item:
relation of right gyration I},
\begin{eqnarray}\label{eqn: in proof lemma bigyration relation in bigyrogroup}
\begin{split}
{} &{}\rgyr{b, a}{}\circ\rgyr{\lgyr{a, b}{c}, a\oplus_b b}{}\\
{} &{\hskip0.3cm}= (\rgyr{a\oplus_b b, \lgyr{a,
b}{c}}{}\circ\rgyr{a,
b}{})^{-1}\\
{} &{\hskip0.3cm}= (\rgyr{\rgyr{b, c}{a}, b\oplus_b
c}{}\circ\rgyr{b,
c}{})^{-1}\\
{} &{\hskip0.3cm}= \rgyr{c, b}{}\circ\rgyr{b\oplus_b c, \rgyr{b,
c}{a}}{}.
\end{split}
\end{eqnarray}
By Identity \eqref{eqn: bigyrodecomposition, commuting relation of
rho and right gyration} and Theorem \ref{thm: bigyrodecomposition
inversion Law}, the extreme sides of \eqref{eqn: in proof lemma
bigyration relation in bigyrogroup} imply
$$\rgyr{c, \rgyr{b, a}{(a\oplus_b b)}}{}\circ\rgyr{b, a}{} = \rgyr{\rgyr{c, b}{(b\oplus_b c)}, a}{}\circ\rgyr{c, b}{}.$$
By Definition \ref{def: bi-gyrogroup operation}, the previous
equation reads
$$\rgyr{c, a\oplus b}{}\circ\rgyr{b, a}{} = \rgyr{b\oplus c, a}{}\circ\rgyr{c, b}{},$$
which completes the proof.
\end{proof}

\begin{thm}[Bi-gyroassociative law in bi-gyrogroups]\label{thm: bigyroassociative law in bigyrogroup}
The bi-gyrogroup $B$ satisfies the left bi-gyroassociative law
\begin{equation}\label{eqn: left bigyroassociative law in bigyrogroup}
a\oplus(b\oplus c) = (a\oplus b)\oplus(\lgyr{a, b}{}\circ\rgyr{b,
a}{})(c)
\end{equation}
and the right bi-gyroassociative law
\begin{equation}\label{eqn: right bigyroassociative law in bigyrogroup}
(a\oplus b)\oplus c = a\oplus (b\oplus(\lgyr{b, a}{}\circ\rgyr{a,
b}{})(c))
\end{equation}
for all $a, b, c\in B$.
\end{thm}
\begin{proof}
From Theorem \ref{thm: bigyroassociative law}, we have $$(a\oplus_b
b)\oplus_b\lgyr{a, b}{c} = \rgyr{b, c}{a}\oplus_b (b\oplus_b c).$$
Applying $\rgyr{c, b}{}$ followed by applying $\rgyr{b\oplus c,
a}{}$ to the previous equation gives
\begin{equation}\label{eqn: in proof of thm bigyroassociative law in bigyrogroup}
(\rgyr{b\oplus c, a}{}\circ\rgyr{c, b}{})((a\oplus_b
b)\oplus_b\lgyr{a, b}{c}) = a\oplus(b\oplus c).
\end{equation}
On the other hand, we compute
\begin{align}\label{eqn: in proof of thm bigyroassociative law in bigyrogroup II}
{} &{} (a\oplus b)\oplus(\lgyr{a, b}{}\circ\rgyr{b, a}{})(c)\notag\\
{} &\hskip5pt = (a\oplus b)\oplus(\rgyr{b, a}{}\circ\lgyr{a, b}{})(c)\notag\\
{} &\hskip5pt= [\rgyr{b, a}{(a\oplus_b
b)}]\oplus[\rgyr{b,a}{(\lgyr{a,
b}{c})}]\notag\\
{} &\hskip5pt= \rgyr{b, a}{((a\oplus_b b)\oplus \lgyr{a, b}{c})}\notag\\
{} &\hskip5pt= (\rgyr{b, a}{}\circ\rgyr{\lgyr{a, b}{c}, a\oplus_b
b}{})((a\oplus_b b)\oplus_b \lgyr{a, b}{c})\notag\\
{} &\hskip5pt= (\rgyr{c, \rgyr{b, a}{(a\oplus_b b)}}{}\circ\rgyr{b,
a}{})((a\oplus_b b)\oplus_b \lgyr{a, b}{c})\notag\\
{} &\hskip5pt= (\rgyr{c, a\oplus b}{}\circ\rgyr{b, a}{})((a\oplus_b
b)\oplus_b \lgyr{a, b}{c}).
\end{align}
We obtain the first equation from Theorem \ref{thm: commuting
relation conjugation and gyration automorphism} \eqref{item: left
and right gyration commute}; the third equation from \eqref{eqn:
composite of right bigyration preserve +}; the fifth equation from
Identity \eqref{eqn: bigyrodecomposition, commuting relation of rho
and right gyration} and Theorem \ref{thm: bigyrodecomposition
inversion Law}.

\par By the lemma, $\rgyr{b\oplus c, a}{}\circ\rgyr{c, b}{} = \rgyr{c,
a\oplus b}{}\circ\rgyr{b, a}{}$. Hence, \eqref{eqn: in proof of thm
bigyroassociative law in bigyrogroup} and \eqref{eqn: in proof of
thm bigyroassociative law in bigyrogroup II} together imply
$a\oplus(b\oplus c) = (a\oplus b)\oplus(\lgyr{a, b}{}\circ\rgyr{b,
a}{})(c)$.
\par Replacing $c$ by $(\lgyr{b, a}{}\circ\rgyr{a, b}{})(c)$ in
\eqref{eqn: left bigyroassociative law in bigyrogroup} followed by
commuting $\lgyr{b, a}{}$ and $\rgyr{a, b}{}$ gives \eqref{eqn:
right bigyroassociative law in bigyrogroup}.
\end{proof}

\begin{thm}[Left gyration reduction property of bi-gyrogroups]\label{thm: left gyration reduction property II}
The bi-gyrogroup $B$ has the left gyration left reduction property
\begin{equation}\label{eqn: left gyration left reduction}
\lgyr{a, b}{} = \lgyr{a\oplus b, b}{}
\end{equation}
and the left gyration right reduction property
\begin{equation}\label{eqn: left gyration right reduction}
\lgyr{a, b}{} = \lgyr{a, b\oplus a}{}
\end{equation}
for all $a, b\in B$.
\end{thm}
\begin{proof}
From \eqref{eqn: left gyration reduction property II}, \eqref{eqn:
left gyration invariant relation} with $\rho = \rgyr{b, a}{}$, and
Theorem \ref{thm: bigyrodecomposition inversion Law}, we have the
following series of equations
\begin{align}
\lgyr{a, b}{} &= \lgyr{a\oplus_b b, \rgyr{a, b}{b}}{}\notag\\
{} &= \lgyr{\rgyr{b, a}{(a\oplus_b b)}, \rgyr{b, a}{(\rgyr{a,
b}{b})}}{}\notag\\
{} &= \lgyr{a\oplus b, b}{},\notag
\end{align}
thus proving \eqref{eqn: left gyration left reduction}. One obtains
similarly that
\begin{align}
\lgyr{a, b}{} &= \lgyr{\rgyr{b, a}{a}, b\oplus_b a}{}\notag\\
{} &= \lgyr{\rgyr{a, b}{(\rgyr{b, a}{a})}, \rgyr{a, b}{(b\oplus_b
a)}}{}\notag\\
{} &= \lgyr{a, b\oplus a}{}.\notag
\end{align}
\end{proof}

\begin{thm}[Right gyration reduction property of bi-gyrogroups]\label{thm: right gyration reduction property II}
The bi-gyrogroup $B$ satisfies the right gyration left reduction
property
\begin{equation}\label{eqn: right gyration left reduction}
\rgyr{a, b}{} = \rgyr{a\oplus b, b}{}
\end{equation}
and the right gyration right reduction property
\begin{equation}\label{eqn: right gyration right reduction}
\rgyr{a, b}{} = \rgyr{a, b\oplus a}{}
\end{equation}
for all $a, b\in B$.
\end{thm}
\begin{proof}
From \eqref{eqn: bigyration reduction property II}, \eqref{eqn:
bigyrodecomposition, commuting relation of rho and right gyration}
with $\rho = \rgyr{b, a}{}$, and Theorem \ref{thm:
bigyrodecomposition inversion Law}, we have the following series of
equations
\begin{eqnarray}\label{eqn: in proof theorem right gyration reduction property}
\begin{split}
\id{B} &= \rgyr{b, a}{}\circ\rgyr{a\oplus_b b, \rgyr{a, b}{b}}{}\\
{} &= \rgyr{\rgyr{b, a}{(a\oplus_b b)}, \rgyr{b, a}{(\rgyr{a,
b}{b})}}{}\circ\rgyr{b, a}{}\\
{} &= \rgyr{a\oplus b, b}{}\circ\rgyr{b, a}{}.
\end{split}
\end{eqnarray}
Hence, the extreme sides of \eqref{eqn: in proof theorem right
gyration reduction property} imply $\rgyr{a, b}{} = \rgyr{a\oplus b,
b}{}$. Applying the bi-gyration inversion law to \eqref{eqn: right
gyration left reduction} followed by interchanging $a$ and $b$ gives
\eqref{eqn: right gyration right reduction}.
\end{proof}

\par Let $(B, \oplus)$ be the corresponding bi-gyrogroup of a
bi-gyrodecomposition $\Gam = H_LBH_R$. By Theorem \ref{thm:
automorphism of bigyrogroup operation}, left and right gyrations of
$(B, \oplus_b)$ preserve the bi-gyrogroup operation. This result and
Theorem \ref{thm: bigyroassociative law in bigyrogroup} motivate the
following definition.

\begin{defn}[Gyration of bi-gyrogroups]\label{def: gyrations of bigyrogroup}
Let $\Gam = H_LBH_R$ be a bi-gyrodecomposition and let $(B,\oplus)$
be the corresponding bi-gyrogroup. The \textit{gyrator} is the map
$$\mathrm{gyr}\colon B\times B\to \aut{B, \oplus}$$
defined by
\begin{equation}\label{eqn: gyrator of bigyrogroup}
\gyr{a, b}{} = \lgyr{a, b}{}\circ\rgyr{b, a}{}
\end{equation}
for all $a, b\in B$.
\end{defn}

\begin{thm}\label{thm: gyrations of bigyrogroup}
For all $a, b\in B$, $\gyr{a, b}{}$ is an automorphism of the
bi-gyrogroup $B$.
\end{thm}
\begin{proof}
The theorem follows from Theorem \ref{thm: automorphism of
bigyrogroup operation}.
\end{proof}

\begin{thm}[Gyroassociative law in bi-gyrogroups]\label{thm: gyroassociative law in bigyrogroup}
The bi-gyrogroup $B$ satisfies the left gyroassociative law
\begin{equation}\label{eqn: left gyroassociative law in bigyrogroup}
a\oplus(b\oplus c) = (a\oplus b)\oplus\gyr{a, b}{c}
\end{equation}
and the right gyroassociative law
\begin{equation}\label{eqn: right gyroassociative law in bigyrogroup}
(a\oplus b)\oplus c = a\oplus (b\oplus\gyr{b, a}{c})
\end{equation}
for all $a, b, c\in B$.
\end{thm}
\begin{proof}
The theorem follows directly from Theorem \ref{thm:
bigyroassociative law in bigyrogroup} and Definition \ref{def:
gyrations of bigyrogroup}.
\end{proof}

\begin{thm}[Gyration reduction property in bi-gyrogroups]\label{thm: gyration reduction property in bigyrogroup}
The bi-gyrogroup $B$ has the left reduction property
\begin{equation}\label{eqn: left reduction property in bigyrogroup}
\gyr{a, b}{} = \gyr{a\oplus b, b}{}
\end{equation}
and the right reduction property
\begin{equation}\label{eqn: right reduction property in bigyrogroup}
\gyr{a, b}{} = \gyr{a, b\oplus a}{}
\end{equation}
for all $a, b\in B$.
\end{thm}
\begin{proof}
From \eqref{eqn: left gyration left reduction} and \eqref{eqn: right
gyration right reduction}, we have the following series of equations
\begin{align}
\gyr{a\oplus b, b}{} &= \lgyr{a\oplus b, b}{}\circ\rgyr{b, a\oplus
b}{}\notag\\
{} &= \lgyr{a,b}{}\circ\rgyr{b, a}{}\notag\\
{} &= \gyr{a, b}{},\notag
\end{align}
thus proving \eqref{eqn: left reduction property in bigyrogroup}.
Similarly, \eqref{eqn: left gyration right reduction} and
\eqref{eqn: right gyration left reduction} together imply
\eqref{eqn: right reduction property in bigyrogroup}.
\end{proof}

\par Theorems \ref{thm: gyroassociative law in bigyrogroup} and
\ref{thm: gyration reduction property in bigyrogroup} indicate that
any bi-gyrogroup is indeed a gyrogroup. Therefore, we recall the
following definition of a gyrogroup.

\begin{defn}[Gyrogroup, \cite{AU2008AHG}]\label{def: gyrogroup}  A groupoid $(G, \oplus)$ is a
\textit{gyrogroup} if its binary operation satisfies the following
axioms.
\begin{enumerate}
    \item [(G1)] There is an element $0\in G$ such that $0\oplus a =
    a$ for all $a\in G$.
    \item [(G2)] For each $a\in G$, there is an element $b\in G$ such that $b\oplus a = 0$.
    \item [(G3)] For all $a, b$ in $G$, there is an automorphism $\gyr{a,b}{}\in \aut{G,\oplus}$ such that
    $$ a\oplus (b\oplus c) = (a\oplus b)\oplus\gyr{a,
    b}{c}
    $$
    for all $c\in G$.
    \item [(G4)] For all $a, b$ in $G$, $\gyr{a, b}{} = \gyr{a\oplus b, b}{}$.
\end{enumerate}
\end{defn}

\begin{defn}[Gyrocommutative gyrogroup, \cite{AU2008AHG}]\label{def: gyrocommutative gyrogroup}
A gyrogroup $(G,\oplus)$ is \textit{gyrocommutative} if it satisfies
the gyrocommutative law
$$a\oplus b = \gyr{a, b}({b\oplus a})$$
for all $a, b\in G$ .
\end{defn}

\begin{thm}\label{thm: bigyrogroup is a gyrogroup}
Let $\Gam = H_LBH_R$ be a bi-gyrodecomposition and let $(B, \oplus)$
be the corresponding bi-gyrogroup. Then $B$ equipped with the
bi-gyrogroup operation forms a gyrogroup.
\end{thm}
\begin{proof}
Axioms (G1) and (G2) are validated in Proposition \ref{prop:
identity and inverse in bigyrogroup}. Axiom (G3) is validated in
Theorems \ref{thm: gyrations of bigyrogroup} and \ref{thm:
gyroassociative law in bigyrogroup}. Axiom (G4) is validated in
Theorem \ref{thm: gyration reduction property in bigyrogroup}.
\end{proof}

\begin{defn}\label{def: bigyrocommutative bigyrodecomposition}
A bi-gyrodecomposition $\Gam = H_LBH_R$ is
\textit{bi-gyrocommutative} if its bi-transversal groupoid is
bi-gyrocommutative in the sense of Definition \ref{def:
Bi-gyrocommutative bi-groupoid}.
\end{defn}

\begin{thm}\label{thm: bigyrocommutative bigyrogroup is a gyrocomm. gyrogroup}
If $\Gam = H_LBH_R$ is a bi-gyrocommutative bi-gyrodecomposition,
then $B$ equipped with the bi-gyrogroup operation is a
gyrocommutative gyrogroup.
\end{thm}
\begin{proof}
Let $a, b\in B$. We compute
\begin{align}
a\oplus b &= \rgyr{b, a}{(a\oplus_b b)}\notag\\
{} &= \rgyr{b, a}{(\lgyr{a, b}{}\circ\rgyr{a, b}{(b\oplus_b
a)})}\notag\\
{} &= (\lgyr{a, b}{}\circ\rgyr{b, a})(\rgyr{a, b}{(b\oplus_b
a)})\notag\\
{} &= \gyr{a, b}{(b\oplus a)},\notag
\end{align}
thus proving $B$ satisfies the gyrocommutative law.
\end{proof}

\par We close this section by proving that having a bi-gyrodecomposition
is an invariant property of groups.

\begin{thm}\label{thm: bigyrodecomposition as invariant}
Let $\Gam_1$ and $\Gam_2$ be isomorphic groups via an isomorphism
$\phi$. If $\Gam_1 = H_LBH_R$ is a bi-gyrodecomposition, then so is
$\Gam_2 = \phi(H_L)\phi(B)\phi(H_r)$.
\end{thm}
\begin{proof}
The proof of this theorem is straightforward, using the fact that
$\phi$ is a group isomorphism from $\Gam_1$ to $\Gam_2$.
\end{proof}

\begin{thm}\label{thm: bigyrocomm. bigyrodecomposition as invariant}
Let $\Gam_1$ and $\Gam_2$ be isomorphic groups via an isomorphism
$\phi$. If $\Gam_1 = H_LBH_R$ is a bi-gyrocommutative
bi-gyrodecomposition, then so is $\Gam_2 =
\phi(H_L)\phi(B)\phi(H_r)$.
\end{thm}
\begin{proof}
This theorem follows from the fact that
\begin{align}
\rgyr{\phi(b_1), \phi(b_2)}{\phi(b)} &= \phi(\rgyr{b_1,
b_2}{b})\notag\\
\lgyr{\phi(b_1), \phi(b_2)}{\phi(b)} &= \phi(\lgyr{b_1,
b_2}{b})\notag
\end{align}
for all $b_1, b_2\in B$.
\end{proof}

\begin{thm}\label{thm: isomorpic gyrogroup from bigyrogroup}
Let $\Gam_1$ and $\Gam_2$ be isomorphic groups via an isomorphism
$\phi$ and let $\Gam_1 = H_LBH_R$ be a bi-gyrodecomposition. Then
the bi-gyrogroups $B$ and $\phi(B)$ are isomorphic as gyrogroups via
$\phi$.
\end{thm}
\begin{proof}
By Theorem \ref{thm: bigyrogroup is a gyrogroup}, $B$ forms a
gyrogroup whose gyrogroup operation is given by $a\oplus b =
\rgyr{b, a}{(a\odot_1 b)}$ for all $a, b\in B$, and $\phi(B)$ forms
a gyrogroup whose gyrogroup operation is given by $c\oplus d =
\rgyr{d, c}{(c\odot_2 d)}$ for all $c, d\in \phi(B)$. Let $a, b\in
B$. We compute
\begin{align}
\phi(a\oplus b) &= \phi(\rgyr{b, a}{(a\odot_1 b)})\notag\\
{} &= \rgyr{\phi(b), \phi(a)}{\phi(a\odot_1 b)}\notag\\
{} &= \rgyr{\phi(b), \phi(a)}{(\phi(a)\odot_2\phi(b))}\notag\\
{} &= \phi(a)\oplus\phi(b).\notag
\end{align}
Hence, the restriction of $\phi$ to $B$ acts as a gyrogroup
isomorphism from $B$ to $\phi(B)$.
\end{proof}

\section{Special pseudo-orthogonal groups}\label{sec: pseudo-orthogonal group}
\par In this section, we provide a concrete realization of a
bi-gyrocommutative bi-gyrodecomposition.

\par A pseudo-Euclidean space $\R^{m, n}$ of signature $(m, n), m,
n\in\N,$ is an $(m+n)$-dimensional linear space with the
pseudo-Euclidean inner product of signature $(m, n)$. The
\textit{special pseudo-orthogonal group}, denoted by $\so{m, n}$,
consists of all the Lorentz transformations of order $(m, n)$ that
leave the pseudo-Euclidean inner product invariant and that can be
reached continuously from the identity transformation in $\R^{m,
n}$. Denote by $\so{m}$ the group of $m\times m$ special orthogonal
matrices and by $\so{n}$ the group of $n\times n$ special orthogonal
matrices.

\par Following \cite{AU2015PRL}, $\so{m}$ and $\so{n}$ can be embedded
into $\so{m, n}$ as subgroups by defining
\begin{eqnarray}\label{map: so(m) and so(n)-->so(m, n)}
\rho\colon O_m &\mapsto& \begin{pmatrix}O_m & 0_{m, n}\\0_{n, m} &
I_n\end{pmatrix},\quad O_m\in\so{m},\\[5pt]
\lamb\colon O_n &\mapsto& \begin{pmatrix}I_m & 0_{m, n}\\0_{n, m} &
O_n\end{pmatrix},\quad O_n\in\so{n}.
\end{eqnarray}
Let $\bet$ be the map defined on the space $\R^{n\times m}$ of all
$n\times m$ real matrices by
\begin{equation}
\bet\colon P \mapsto \begin{pmatrix}\sqrt{I_m+P^{\textsf{t}}P} & P^{\textsf{t}}\\
P & \sqrt{I_n+PP^{\textsf{t}}}\end{pmatrix},\quad P\in\R^{n\times
m}.
\end{equation}
It is easy to see that $\bet$ is a bijection from $\R^{n\times m}$
to $\bet(\R^{n\times m})$.

\par Note that
\begin{align}
\rho(\so{m}) &= \BCset{\begin{pmatrix}O_m & 0_{m, n}\\0_{n, m} &
I_n\end{pmatrix}}{O_m\in\so{m}}\notag\\[5pt]
\lamb(\so{n}) &= \BCset{ \begin{pmatrix}I_m & 0_{m, n}\\0_{n, m} &
O_n\end{pmatrix}}{O_n\in\so{n}}\notag\\[5pt]
\bet(\R^{n\times m}) &=
\BCset{\begin{pmatrix}\sqrt{I_m+P^{\textsf{t}}P} & P^{\textsf{t}}\\
P & \sqrt{I_n+PP^{\textsf{t}}}\end{pmatrix}}{P\in\R^{n\times
m}}.\notag
\end{align}

\par It follows from Examples 22 and 23 of \cite{AU2015PRL} that
$\lamb(\so{n})$ and $\rho(\so{m})$ are subgroups of $\so{m, n}$.
Further, $\so{m}$ and $\rho(\so{m})$ are isomorphic as groups via
$\rho$, and $\so{n}$ and $\lamb(\so{n})$ are isomorphic as groups
via $\lamb$.

\par We will see shortly that
$$\so{m, n} = \rho(\so{m})\bet(\R^{n\times
m})\lamb(\so{n})$$ is a bi-gyrocommutative bi-gyrodecomposition.

\par By Theorem 8 of \cite{AU2015PRL}, $\bet(\R^{n\times m})$
is a bi-transversal of subgroups $\rho(\so{m})$ and $\lamb(\so{n})$
in the pseudo-orthogonal group $\so{m, n}$. From Lemma 6 of
\cite{AU2015PRL}, we have
\begin{align}
\rho(O_m)\bet(P)\rho(O_m)^{-1} &= \bet(PO_m^{-1})\notag\\
\lamb(O_n)\bet(P)\lamb(O_n)^{-1} &= \bet(O_nP)\notag
\end{align}
for all $O_m\in\so{m}$, $O_n\in\so{n}$, and $P\in\R^{n\times m}$.
Hence, $\rho(\so{m})$ and $\lamb(\so{n})$ normalize
$\bet(\R^{n\times m})$. Setting $P = 0_{n, m}$ in the third identity
of (77) of \cite{AU2015PRL}, we have
$$\lamb(O_n)\rho(O_m) = \rho(O_m)\lamb(O_n)$$
for all $O_m\in\so{m}, O_n\in\so{n}$ because $\bet(P) = \bet(0_{n,
m}) = I_{m+n}$. Thus, $\bet(\R^{n\times m})$ is a bi-gyrotransversal
of $\rho(\so{m})$ and $\lamb(\so{n})$ in $\so{m, n}$.

\par In Theorem 13 of \cite{AU2015PRL}, the bi-gyroaddition, $\oplus_U$,
and bi-gyrations in the \mbox{parameter} bi-gyrogroupoid
$\R^{n\times m}$ are given by

\begin{align}
P_1\oplus_U P_2 &= P_1\sqrt{I_m+P_2^\textsf{t}P_2} +
\sqrt{I_n+P_1P_1^\textsf{t}}P_2\notag\\
\lgyr{P_1, P_2}{} &= \sqrt{I_n+P_{1,2}P_{1,2}^\textsf{t}}^{\hskip2pt-1}\Bset{P_1P_2^\textsf{t}+\sqrt{I_n+P_1P_1^\textsf{t}}\sqrt{I_n+P_2P_2^\textsf{t}}}\notag\\
\rgyr{P_1, P_2}{} &=
\Bset{P_1^\textsf{t}P_2+\sqrt{I_m+P_1^\textsf{t}P_1}\sqrt{I_m+P_2^\textsf{t}P_2}}\sqrt{I_m+P_{1,2}^\textsf{t}P_{1,2}}^{\hskip2pt-1}\notag
\end{align}
for all $P_1, P_2\in\R^{n\times m}$ and $P_{1,2} = P_1\oplus_U P_2$.

\par From (74) of \cite{AU2015PRL}, we have $I_{m+n} = B(0_{n,
m})\in\bet(\R^{n\times m})$. From Theorem 10 of \cite{AU2015PRL}, we
have $\bet(P)^{-1} = \bet(-P)\in\bet(\R^{n\times m})$ for all $P\in
\R^{n\times m}$. From Equations (179) and (184) of \cite{AU2015PRL},
we have
$$\bet(P_1)\bet(P_2)\bet(P_1) = \bet((P_1\oplus_U P_2)\oplus_U\lgyr{P_1, P_2}{P_1}).$$
Hence, $\bet(P_1)\bet(P_2)\bet(P_1)\in\bet(\R^{n\times m})$ for all
$P_1, P_2\in\R^{n\times m}$. This proves that $\bet(\R^{n\times m})$
is a twisted subgroup of $\so{m, n}$.

\par By (104) of \cite{AU2015PRL},
\begin{equation}\label{eqn: product of two biboost}
\bet(P_1)\bet(P_2) = \rho(\rgyr{P_1, P_2}{})\bet(P_1\oplus_U
P_2)\lamb(\lgyr{P_1, P_2}{})
\end{equation}
for all $P_1, P_2\in\R^{n\times m}$. Hence, the left and right
transversal maps induced by the decomposition $\so{m, n} =
\rho(\so{m})\bet(\R^{n\times})\lamb(\so{n})$ are given by
\begin{equation}\label{eqn:  left bi-transversal map in R(n x m)}
h_\ell(\bet(P_1), \bet(P_2)) = \rho(\rgyr{P_1, P_2}{})
\end{equation}
and
\begin{equation}\label{eqn:  right bi-transversal map in R(n x m)}
h_r(\bet(P_1), \bet(P_2)) = \lamb(\lgyr{P_1, P_2}{})
\end{equation}
for all $P_1, P_2\in\R^{n\times m}$.

\par By (162b) of \cite{AU2015PRL},
$\irgyr{P_1, P_2}{} = \rgyr{P_2, P_1}{}$. Hence,
$$h_\ell(\bet(P_1), \bet(P_2))^{-1} = \rho(\irgyr{P_1, P_2}{}) = \rho(\rgyr{P_2, P_1}{}) = h_\ell(\bet(P_2),
\bet(P_1)).$$ Similarly, (162a) of \cite{AU2015PRL} implies
$h_r(\bet(P_1), \bet(P_2))^{-1} = h_r(\bet(P_2), \bet(P_1))$.
\mbox{Combining} these results gives
\begin{thm}\label{thm: SO(m, n) is bigyrodecomposition}
The decomposition
\begin{equation}\label{eqn: decomposition of SO(m, n)}
\so{m, n} = \rho(\so{m})\bet(\R^{n\times m})\lamb(\so{n})
\end{equation}
is a bi-gyrodecomposition.
\end{thm}

\par By \eqref{eqn: product of two biboost}, the
bi-transversal operation induced by the decomposition \eqref{eqn:
decomposition of SO(m, n)} is given by
\begin{equation}\label{eqn: explicit formula for bitranversal operation in B(P)}
\bet(P_1)\oplus_b\bet(P_2) = \bet(P_1\oplus_U P_2)
\end{equation}
for all $P_1, P_2\in\R^{n\times m}$.

\par Note that $\rgyr{P_1, P_2}{}$ is an $m\times
m$ matrix and $\lgyr{P_1, P_2}{}$ is an $n\times n$ matrix, while
$\rgyr{\bet(P_1), \bet(P_2)}{}$ and $\lgyr{\bet(P_1), \bet(P_2)}{}$
are maps. By \eqref{eqn: bigyration in conjugate form}, the action
of left and right gyrations on $\bet(\R^{n\times n})$ is given by
\begin{equation}\label{eqn: left gyration on B(P)}
\lgyr{\bet(P_1), \bet(P_2)}{\bet(P)} = \bet(\lgyr{P_1, P_2}{P})
\end{equation}
and
\begin{equation}\label{eqn: right gyration on B(P)}
\rgyr{\bet(P_1), \bet(P_2)}{\bet(P)} = \bet(P\rgyr{P_1, P_2}{})
\end{equation}
for all $P_1, P_2, P\in\R^{n\times m}$. Using \eqref{eqn: left
gyration on B(P)} and \eqref{eqn: right gyration on B(P)}, together
with Theorem 25 of \cite{AU2015PRL}, we have
\begin{thm}\label{thm: SO(m, n) is bigyrocommutative}
The bi-gyrodecomposition $$\so{m, n} = \rho(\so{m})\bet(\R^{n\times
m})\lamb(\so{n})$$ is bi-gyrocommutative.
\end{thm}

\par By Theorem 52 of \cite{AU2015PRL}, the space $\R^{n\times m}$ of all $n\times m$ real
matrices forms a gyrocommutative gyrogroup under the operation
$\oplus'_U$ given by
\begin{equation}\label{eqn: bigyrogroup bigyroaddition in R(n x m)}
P_1\oplus'_U P_2 = (P_1\oplus_U P_2)\rgyr{P_2, P_1}{},\quad P_1,
P_2\in\R^{n\times m}.
\end{equation}

\begin{thm}\label{thm: B(P) is a gyrocomm. gyrogroup}
The set
$$\bet(\R^{n\times m}) = \BCset{\begin{pmatrix}\sqrt{I_m+P^{\textsf{t}}P} & P^{\textsf{t}}\\
P & \sqrt{I_n+PP^{\textsf{t}}}\end{pmatrix}}{P\in\R^{n\times m}}$$
together with the bi-gyrogroup operation $\oplus$ given by
$$\bet(P_1)\oplus\bet(P_2) = \bet((P_1\oplus_U P_2)\rgyr{P_2, P_1}{})$$
is a gyrocommutative gyrogroup isomorphic to $(\R^{n\times m},
\oplus'_U)$.
\end{thm}
\begin{proof}
The theorem follows from Theorems \ref{thm: SO(m, n) is
bigyrodecomposition}, \ref{thm: SO(m, n) is bigyrocommutative},
\ref{thm: bigyrogroup is a gyrogroup}, and \ref{thm:
bigyrocommutative bigyrogroup is a gyrocomm. gyrogroup}. Further,
the bi-gyrogroup operation $\oplus$ is given by
\begin{align}
\bet(P_1)\oplus\bet(P_2) &= \rgyr{\bet(P_2),
\bet(P_1)}{(\bet(P_1)\oplus_b\bet(P_2))}\notag\\
{} &= \rgyr{\bet(P_2), \bet(P_1)}{\bet(P_1\oplus_U P_2)}\notag\\
{} &= \bet((P_1\oplus_U P_2)\rgyr{P_2, P_1}{}).\notag
\end{align}
From \eqref{eqn: bigyrogroup bigyroaddition in R(n x m)}, we have
$\bet(P_1)\oplus\bet(P_2) = \bet(P_1\oplus'_U P_2)$. Hence, $\bet$
acts as a gyrogroup isomorphism from $\R^{n\times m}$ to
$\bet(\R^{n\times m})$.
\end{proof}

\section{Spin groups}\label{sec: spin group}
We establish that the spin group of the Clifford algebra of
pseudo-Euclidean space $\R^{m, n}$ of signature $(m, n)$ has a
bi-gyrocommutative bi-gyrodecomposition. For the basic notion of
Clifford algebras, the reader is referred to \cite{GGMM1991CAD,
PL2001CAS, BLMM1989SG, LG2001CGG}.

\par Let $(V, B)$ be a real quadratic space. That is, $V$ is a linear
space over $\R$ \mbox{together} with a non-degenerate symmetric
bilinear form $B$. Let $Q$ be the \mbox{associated} quadratic form
given by $Q(v) = B(v, v)$ for $v\in V$. Denote by $\CL{V, Q}$ the
\textit{\mbox{Clifford} algebra} of $(V, B)$. Define
\begin{equation}\label{eqn: Clifford group}
\Cg{V, Q} = \cset{g\in\mulCL{V, Q}}{\forall v\in V,\hskip2pt
\hat{g}vg^{-1}\in V}.
\end{equation}
Here, $\hat\cdot$ stands for the unique involutive automorphism of
$\CL{V, Q}$ such that $\hat{v} = -v$ for all $v\in V$, known as the
\textit{grade involution}. If $V$ is \textit{finite} dimensional,
then $\Cg{V, Q}$ is indeed a subgroup of the group of units of
$\CL{V, Q}$, called the \textit{Clifford group of $\CL{V, Q}$}. In
this case, any element $g$ of $\Cg{V, Q}$ induces the linear
automorphism $T_g$ of $V$ given by
\begin{equation}
T_g(v) = \hat{g}vg^{-1},\quad v\in V.
\end{equation}

\par Since $T_g\circ T_h = T_{gh}$ for all $g, h\in \Cg{V, Q}$,
the map $\pi\colon g\mapsto T_g$ defines a group homomorphism from
$\Cg{V, Q}$ to the general linear group $\GL{V}$, known as the
\textit{twisted adjoint representation of $\Cg{V, Q}$}. The kernel
of $\pi$ equals $\R^\times 1 := \cset{\lamb1}{\lamb\in\R, \lamb\ne
0}$. By the Cartan-Dieudonn\'{e} theorem, $\pi$ maps $\Cg{V, Q}$
onto the orthogonal group $\Or{V, Q}$.

\par Recall that, in the Clifford algebra $\CL{V, Q}$, we have $v^2 =
Q(v)1$ for all $v\in V$. Hence, if $v\in V$ and $Q(v)\ne 0$, then
$v$ is invertible whose inverse is $v/Q(v)$. Further, we have an
important identity $uv+vu = 2B(u, v)1$ for all $u, v\in V$. Using
this identity, we obtain
$$-vuv^{-1} = u - (uv + vu)v^{-1} = u - (2B(u, v)1)\Bp{\frac{v}{Q(v)}}
= u - \frac{2B(u, v)}{Q(v)}v,$$ which implies $\hat{v}uv^{-1} =
-vuv^{-1}\in V$ for all $u\in V$. Hence, if $v\in V$ and $Q(v)\ne
0$, then $v\in\Cg{V, Q}$. In fact, $T_v$ is \textit{the reflection
about the hyperplane orthogonal to $v$}. We also have the following
important subgroup of the Clifford group of $\CL{V, Q}$:
\begin{equation}\label{eqn: spin group}
\spin{V, Q} = \cset{v_1v_2\cdots v_r}{r\textrm{ is even, }v_i\in V,
\textrm{ and } Q(v_i) = \pm1},
\end{equation}
known as the \textit{spin group of $\CL{V, Q}$}.

\par The following theorem is well known in the literature.
Its proof can be found, for instance, in Theorem 2.9 of
\cite{BLMM1989SG}.

\begin{thm}\label{thm: pi maps spin to SO(m, n)}
The restriction of the twisted adjoint representation to the spin
group of $\CL{V, Q}$ is a surjective group homomorphism from
$\spin{V, Q}$ to the special orthogonal group $\so{V, Q}$ of $V$.
Its kernel is $\set{1, -1}$.
\end{thm}

\begin{cor}\label{cor: the quotient of spin isomorphic to SO(V, Q)}
The quotient group $\spin{V, Q}/\set{1, -1}$ and the special
orthogonal group $\so{V, Q}$ are isomorphic.
\end{cor}

\par As $V$ is a linear space over $\R$, we can choose an ordered
basis for $V$ so that
$$Q(v) = v_1^2+v_2^2+\cdots+v_m^2-v_{m+1}^2-v_{m+2}^2-\cdots -v_{m+n}^2$$
for all $v = (v_1,\dots, v_m, v_{m+1}, \dots, v_{m+n})\in \R^{m+n}$,
\cite[Theorem 4.5]{LG2001CGG}. Hence, $\so{V, Q}\equiv\so{m, n}$ and
$\spin{V, Q}\equiv\spin{m, n}$. Corollary \ref{cor: the quotient of
spin isomorphic to SO(V, Q)} implies that
\begin{equation}\label{eqn: SO(m, n) isomorphic to the quotient of spin}
\spin{m, n}/\set{1, -1}\cong\so{m, n}.
\end{equation}
Hence, we have the following theorem.

\begin{thm}\label{thm: SO(m, n) has bigyrodecomposition}
The quotient group $$\spin{m, n}/\set{1, -1}$$ has a
bi-gyrocommutative bi-gyrodecomposition.
\end{thm}
\begin{proof}
This theorem follows directly from \eqref{eqn: SO(m, n) isomorphic
to the quotient of spin} and Theorems \ref{thm: bigyrocomm.
bigyrodecomposition as invariant} and \ref{thm: SO(m, n) is
bigyrocommutative}.
\end{proof}

\section{Conclusion}\label{sec: conclusion}
\par A gyrogroup is a non-associative group-like structure in which
the non-associativity is controlled by a special family of
automorphisms called gyrations. Gyrations, in turn, result from the
extension by abstraction of the relativistic effect known as
\textit{Thomas precession}. In this paper we generalize the notion
of gyrogroups, which involves a single family of gyrations, to that
of bi-gyrogroups, which involves two distinct families of gyrations,
collectively called bi-gyrations.

\par The bi-transversal decomposition $\Gam = H_LBH_R$,
studied in Section \ref{sec: bitransversal}, naturally leads to a
groupoid $(B, \odot)$ that comes with two families of automorphisms,
left and right ones. This groupoid is related to the bi-gyrogroupoid
$(B, \oplus_b)$, studied earlier in Section \ref{sec:
bigyrogroupoid}. Bi-gyrogroupoids $(B, \oplus_b)$ form an
intermediate structure that suggestively leads to the desired
bi-gyrogroup structure $(B, \oplus)$. The bi-transversal operation
$\odot$ arises naturally from the bi-transversal decomposition
\eqref{eqn: Unique decomposition}. Under the natural conditions of
Definition \ref{def: Bi-gyrodecomposition}, the bi-transversal
operation $\odot$ becomes the bi-gyrogroupoid operation $\oplus_b$.
The latter operation leads to the desired bi-gyrogroup operation
$\oplus$ by means of \eqref{eqn: bigyrogroup operation}.

\par As we have shown in Section \ref{sec: bigyrodecomposition},
any bi-gyrodecomposition $\Gam = H_LBH_R$ of a group $\Gam$ induces
the bi-gyrogroup structure on $B$, giving rise to a bi-gyrogroup
$(B, \oplus)$ along with left gyrations $\lgyr{a, b}{}$ and right
gyrations $\rgyr{a, b}{}$, $a, b\in B$. Further, in the case where
$H_L$ is the trivial subgroup of $\Gam$, the bi-gyrodecomposition
reduces to the decomposition $\Gam = BH$ studied in
\cite{TFAU2000IDG}. The bi-gyrogroup $(B, \oplus)$ induced by a
bi-gyrodecomposition of a group is indeed an abstract version of the
bi-gyrogroup $\R^{n\times m}$ of all $n\times m$ real matrices
studied in \cite{AU2015PRL}.

\par Bi-gyrogroups are group-like structures. For instance, they satisfy
the bi-gyroassociative law (Theorem \ref{thm: bigyroassociative law
in bigyrogroup}), which descends to the associative law if their
left and right gyrations are the identity automorphism. A concrete
realization of a bi-gyrogroup is found in the special
pseudo-orthogonal group $\so{m, n}$ of the pseudo-Euclidean space
$\R^{m, n}$ of signature $(m, n)$, as shown in \cite{AU2015PRL} and
in Section \ref{sec: pseudo-orthogonal group}. Moreover,
bi-gyrogroups arise in the group counterpart of Clifford algebras as
we establish in Section \ref{sec: spin group} that the quotient
group $\spin{m, n}/\set{1, -1}$ of the spin group possesses a
bi-gyrodecomposition.

\par By Theorem \ref{thm: bigyrogroup is a gyrogroup},
any bi-gyrogroup is a gyrogroup. Yet, in general, the
bi-gyrostructure of a bi-gyrogroup is richer than the gyrostructure
of a gyrogroup. To see this clearly, we note that gyrations $\gyr{a,
b}{}$ of a gyrogroup $(B, \oplus)$, $a, b\in B$, are completely
determined by the gyrogroup operation according to the
\textit{gyrator identity} in Theorem 2.10 (10) of \cite{AU2008AHG}:
\begin{equation}\label{eqn: gyrator identity}
\gyr{a, b}{x} = \ominus(a\oplus b)\oplus(a\oplus (b\oplus x))
\end{equation}
for all $a, b, x$ in the gyrogroup $(B, \oplus)$. In contrast, the
\textit{bi-gyrator identity} \mbox{analogous} to \eqref{eqn: gyrator
identity} is
\begin{equation}\label{eqn: bi-gyrator identity}
(\lgyr{a, b}{}\circ\rgyr{b, a}){(x)} = \ominus(a\oplus
b)\oplus(a\oplus(b\oplus x))
\end{equation}
for all $a, b, x$ in a bi-gyrogroup $(B, \oplus)$. Here, the
bi-gyrogroup operation \mbox{completely} determines the composite
automorphism $\lgyr{a, b}{}\circ\rgyr{b, a}{}$. However, it does not
determine straightforwardly each of the two automorphisms $\lgyr{a,
b}{}$ and $\rgyr{a, b}{}$. Thus, the presence of two families of
gyrations in a bi-gyrogroup, as opposed to the presence of a single
family of gyrations in a gyrogroup, significantly enriches the
bi-gyrostructure of bi-gyrogroups.

\vskip0.5cm
\par\noindent\textbf{Acknowledgments.}
As a visiting researcher, the first author would like to express his
special gratitude to the Department of Mathematics, North Dakota
State University, and his host. This work was completed with the
support of Development and Promotion of Science and Technology
Talents Project (DPST), Institute for Promotion of Teaching Science
and Technology (IPST), Thailand.

\bibliographystyle{amsplain}\addcontentsline{toc}{section}{References}
\bibliography{Bi_Gyrogroups}
\end{document}